\documentclass[reqno, a4paper, 10pt]{amsart}
\usepackage{amssymb,url, color, mathrsfs}
\usepackage[colorlinks=true, bookmarks=true, pdfstartview=FitH, pagebackref=true, linktocpage=true]{hyperref}
\usepackage[short,nodayofweek]{datetime}
\usepackage[pagewise,running,mathlines,displaymath, switch]{lineno}
\usepackage{times}

\theoremstyle{plain}
\numberwithin{equation}{section}
\newtheorem{theorem}{Theorem}
\newtheorem{lemma}{Lemma}
\newtheorem{corollary}{Corollary}
\newtheorem{proposition}{Proposition}
\theoremstyle{remark}

\DeclareMathOperator{\Rset}{\mathbf{R}}
\DeclareMathOperator{\vol}{\textrm{vol}}

\definecolor{brown}{rgb}{0.5,0,0}
\definecolor{backgroundcolor}{rgb}{0.98, 0.92, 0.73}

\newcommand{\Sscr}{\mathscr S}
\newcommand{\Cscr}{\mathscr C}
\newcommand{\Nscr}{\mathscr N}

\newif\ifprint
\ifprint
\else
\fi

\def\R{\mathbb R}

\author[Q.A. Ng\^o]{Qu\^{o}\hspace{-0.5ex}\llap{\raise 1ex\hbox{\'{}}}\hspace{0.5ex}c Anh Ng\^o}
\address[Q.A. Ng\^o]{Department of Mathematics\\
College of Science, Vi\^{e}t Nam National University\\
H\`{a} N\^{o}i, Vi\^{e}t Nam.}
\email{\href{mailto: Q.A. Ng\^o <nqanh@vnu.edu.vn>}{nqanh@vnu.edu.vn}}
\email{\href{mailto: Q.A. Ng\^o <bookworm\_vn@yahoo.com>}{bookworm\_vn@yahoo.com}}

\author[V.H. Nguyen]{Van Hoang Nguyen}
\address[V.H. Nguyen]{School of Mathematical Sciences\\
Tel Aviv University\\
Tel Aviv 69978, Israel.}

\address[V.H. Nguyen]{Institut de Math\'ematiques de Toulouse\\
Universit\'e Paul Sabatier\\
31062 Toulouse c\'edex 09, France.}

\email{\href{mailto: V.H. Nguyen <van-hoang.nguyen@math.univ-toulouse.fr>}{van-hoang.nguyen@math.univ-toulouse.fr}}
\email{\href{mailto: V.H. Nguyen <vanhoang0610@yahoo.com>}{vanhoang0610@yahoo.com}}

\begin{document} 

\allowdisplaybreaks

\setpagewiselinenumbers
\setlength\linenumbersep{100pt}

\title[Sharp reversed Hardy--Littlewood--Sobolev inequality on $\Rset_+^n$]
{Sharp reversed Hardy--Littlewood--Sobolev inequality on the half space $\Rset_+^n$}

\begin{abstract}
This is the second in our series of papers concerning some reversed Hardy--Littlewood--Sobolev inequalities. In the present work, we establish the following sharp reversed Hardy--Littlewood--Sobolev inequality on the half space $\Rset_+^n$
for any nonnegative functions $f\in L^p(\partial \Rset_+^n)$, $g\in L^r(\Rset_+^n)$, and $p,r\in (0,1)$, $\lambda > 0$ such that $(1-1/n)1/p + 1/r -(\lambda-1) /n =2$. Some estimates for $\Cscr_{n,p,r}$ as well as the existence of extrema functions for this inequality are also considered. New ideas are also introduced in this paper. 
\end{abstract}

\date{\bf \today \; at \, \currenttime}



\maketitle


\section{Introduction}

This is the second in our series of papers concerning some reversed Hardy--Littlewood--Sobolev (HLS) inequalities which tell us how to bound $\int \int f(x) |x-y|^\lambda g(y) dxdy$ in terms of $\|f\|_p \|g\|_r$ for suitable numbers $p$ and $r$. In the literature, the classical HLS inequality named after Hardy--Littlewood \cite{hl1928,hl1930} and Sobolev \cite{sobolev1938} concerns the following estimate
\begin{equation}\label{eq:HLSineq}
\int_{\Rset^n} \int_{\Rset^n} \frac{f(x)  g(y)}{|x-y|^{-\lambda}} dx dy \leqslant \Nscr_{n,\lambda,p} \Big( \int_{\Rset^n} |f|^p dx \Big)^{1/p} \Big( \int_{\Rset^n} |g|^r dx \Big)^{1/r} 
\end{equation}
for some constant $\Nscr_{n,\lambda,p} > 0$ and for all $f \in L^p(\Rset^n)$ and $g \in L^r(\Rset^n)$ where the positive constants $p$ and $r$ are related via the following $1/p+1/r-\lambda/n=2$ for some $\lambda <0$ is crucial. As the impact of the HLS inequality on quantitative theories of solutions of (partial) differential equations is now clear, the HLS inequality as well as its variants have captured much attention by many mathematicians.

In one way or another, the HLS inequality and other famous inequalities are related to each other; see \cite{beckner1993}. To see this more precise, it is quite a surprise to remark that the HLS inequality and the Sobolev inequality are indeed dual for certain families of exponents. First, we let $\lambda = n-2s$ in \eqref{eq:HLSineq} and rewrite the right hand side of \eqref{eq:HLSineq} with the fact that ${2^{ - 2s}}{\pi ^{ - n/2}} \Gamma (n/2-s)/\Gamma (s)$ is simply the Green function of the operator $(-\Delta)^s$ in $\Rset^n$ for each $s \in (0,n/2)$ to get
\begin{equation}\label{eqHLS->Sobolev}
\int_{\Rset^n} f (-\Delta)^{-s}(f) dx \leqslant \Sscr_{n,s} \Big( \int_{\Rset^n} |f|^{2n/(n+2s)} dx \Big)^{1+2s/n}.
\end{equation}
Hence, the sharp HLS inequality can imply the sharp Sobolev inequality. Further seminal works also reveal that the sharp HLS inequality can also imply the Moser--Trudinger--Onofri inequality and the logarithmic HLS inequality \cite{beckner1993}, as well as the Gross logarithmic Sobolev inequality \cite{gross1975}. Clearly, all these inequalities have many important applications in analysis and geometry, as well as in quantum field theory. 

Although the HLS inequality \eqref{eq:HLSineq} (not in the sharp form) was proved earlier, it took quite a long time to find the sharp version with the precise sharp constant of \eqref{eq:HLSineq}; see \cite{l1983}. Although Lieb was able to prove the existence of optimizers for \eqref{eq:HLSineq} for any $p$ and $r$, neither the sharp constant nor the precise form of optimizers are known except in the diagonal case $p=r$. In this special case, $p=r=2n/(2n+\lambda)$, the sharp constant $\Nscr_{n,\lambda,p}$ is
\[
\Nscr_{n,\lambda,p} =\Nscr_{n,\lambda} = {\pi ^{\lambda /2}}\frac{{\Gamma (n/2 - \lambda /2)}}{{\Gamma (n -  \lambda /2)}}{\left( {\frac{{\Gamma (n)}}{{\Gamma (n/2)}}} \right)^{1-\lambda /n}} .
\]
As we have already mentioned before, in the last two decades, the classical HLS inequality \eqref{eq:HLSineq} has captured much attention by many mathematicians. Some remarkable extensions and generalizations have already been drawn, for example, one has HLS inequalities on Heisenberg groups, on compact Riemannian manifolds, and on weighted forms; see \cite{fl2012a,hanzhu,st1958} for details. 

Among extensions and generalizations in the literature, let us mention the following two results. The first result concerns a so-called reversed HLS inequality which recently proved by Dou and Zhu in \cite{dz2014} for the case of the whole space $\Rset^n$. By using an extension of the classical Marcinkiewicz interpolation theorem applying to certain singular integral operators, Dou and Zhu established a reversed HLS inequality in the whole space $\Rset^n$ which turns out to be useful when studying some curvature equations with negative critical Sobolev exponents; see \cite{zhu2014}. Their result can be stated as follows.

\begin{theorem}[see \cite{dz2014}]\label{thmDouZhuReversed}
Let $p,r\in (0,1)$ and $\lambda > 0$ such that $1/p + 1/r -\lambda /n =2$. Then there exists a best constant $\Cscr_{n,p,r}>0$ such that for any nonnegative functions $f\in L^p(\Rset^n)$ and $g\in L^r(\Rset^n)$, we have
\begin{equation}\label{eq:RHLS}
\int_{\Rset^n} \int_{\Rset^n} f(x) |x-y|^\lambda g(y) dx dy \geqslant \Cscr_{n,p,r} \Big( \int_{\Rset^n} |f|^p dx \Big)^{1/p} \Big( \int_{\Rset^n} |g|^r dx \Big)^{1/r} .
\end{equation}
\end{theorem}

In the first paper of our series \cite{NgoNguyen2015}, we provided an alternative way to reprove Theorem \ref{thmDouZhuReversed} which is simpler and more direct that the method used in \cite{dz2014}. We also calculated the sharp constant $\Cscr_{n,\lambda}$ in the diagonal case $p=r=2n/(2n+\lambda)$. It it quite interesting to note that the sharp constant $\Cscr_{n,\lambda}$ also takes the same form as $\Nscr_{n,\lambda}$. (However, the value of $\Cscr_{n,\lambda}$ and $\Nscr_{n,\lambda}$ are different since we require $\lambda<0$ in $\Nscr_{n,\lambda}$ and $\lambda >0$ in $\Cscr_{n,\lambda}$.)

The second result that we wish to address is a so-called HLS inequality on the upper half space $\Rset_+^n$ which can be seen as an extension of \eqref{eq:HLSineq} which was established by Dou and Zhu in \cite{dz2013}.

In order to state this result, some notation and conventions are needed. First, given $n \geqslant 3$, by $\Rset_+^n$, we mean the Euclidean half space given by 
$$\Rset_+^n = \{y = (y',y_n) \in \Rset^n : y' \in \Rset^{n-1}, y_n > 0\}$$ 
and by $\partial \Rset_+^n$, we mean the boundary of $\Rset_+^n$; hence we can identify $\partial \Rset_+^n = \Rset^{n-1}$. Upon using the above notations and for the sake of simplicity, for each $y \in \Rset_+^n$ and $x \in \partial \Rset_+^n$, we can write 
$$|x-y| = \sqrt{|y'-x|^2 + y_n^2}$$ 
with, of course, $y=(y',y_n)$. Conventionally and for the sake of clarity, throughout the present work, by $x$ we usually mean a point in $\partial \Rset_+^n$ while by $y$ we mean a point in $\Rset_+^n$.

We are now in a position to state the HLS inequality on the upper half space $\Rset_+^n$ in \cite{dz2013}.

\begin{theorem}[see \cite{dz2013}]\label{thmDouZhu}
For any $n\geqslant 2$, $\lambda \in (1-n, 0)$ and $p,r>1$ satisfying $(n-1)/(np) +1/r - \lambda /n =2-1/n$, there exists a best constant $\Nscr_{n,\alpha,p}^+ > 0$ depending only on $n,\alpha$ and $p$ such that for any nonnegative functions $f\in L^p(\partial \Rset_+^n)$ and $g\in L^r(\Rset_+^n)$, it holds
\begin{equation}\label{eq:DZreverseHLS}
\int_{\partial \Rset_+^n} \int_{\Rset_+^n} \frac{f(x)  g(y)}{|x-y|^{-\lambda}} dy dx \leqslant \Nscr_{n,\alpha,p}^+ \Big( \int_{\partial\Rset_+^n} |f|^p dx \Big)^{1/p} \Big( \int_{\Rset_+^n} |g|^r dy \Big)^{1/r} .
\end{equation}
\end{theorem}

Motivated by Theorems \ref{thmDouZhuReversed} and \ref{thmDouZhu}, the aim of the present paper is to propose a reversed version of the classical HLS inequality on the upper half space $\Rset_+^n$. The following theorem is our main result of the paper.

\begin{theorem}\label{thmMAIN}
For any $n\geqslant 2$, $\lambda > 0$ and $p,r\in (0,1)$ satisfying
\begin{equation}\label{eq:condHLS}
\frac{n-1}n \frac1p + \frac1r -\frac{\lambda}n =2-\frac 1n,
\end{equation}
there exists a best constant $\Cscr_{n,\alpha,p}^+ > 0$ depending only on $n,\alpha$ and $p$ such that for any nonnegative functions $f\in L^p(\partial \Rset_+^n)$ and $g\in L^r(\Rset_+^n)$, there holds
\begin{equation}\label{eq:reverseHLS}
\int_{\Rset_+^n} \int_{\partial \Rset_+^n} f(x) |x-y|^\lambda g(y) dx dy \geqslant \Cscr_{n,\alpha,p}^+ \Big( \int_{\partial\Rset_+^n} |f|^p dx \Big)^{1/p} \Big( \int_{\Rset_+^n} |g|^r dy \Big)^{1/r}.
\end{equation}
\end{theorem}

To prove Theorem \ref{thmMAIN}, we adopt the standard approach, based on the layer cake representation, for the classical HLS inequality on$\Rset^n$; see \cite[Section 4.3]{liebloss2001}. Once we can establish Theorem \ref{thmMAIN}, it is natural to ask whether the extremal functions for the reversed HLS inequality \eqref{eq:RHLS} actually exist. To this purpose, inspired by \cite{dz2013}, let us first introduce an ``Laplacian-type extension" operator $E_\lambda$ for any function $f$ on $\partial \Rset_+^n$ to a function on $\Rset_+^n$ as follows
\[
(E_\lambda f)(y) = \int_{\partial \Rset_+^n} |x-y|^\lambda f(x) dx
\]
for $y\in \Rset_+^n$. (Note that $E_\lambda f$ does not agree with $f$ on $\partial \Rset_+^n$ since $(E_\lambda f)(y',0) = \int_{\partial \Rset_+^n} |x-y'|^\lambda f(x) dx$.) Then the reversed HLS inequality \eqref{eq:reverseHLS} is equivalent to the following inequality
\begin{equation}\label{eq:anotherform}
\Big( \int_{ \Rset_+^n} |E_\lambda f|^q dy \Big)^{1/q} \geqslant \Cscr_{n,\alpha,p}^+ \Big( \int_{\partial \Rset_+^n} |f|^p dx \Big)^{1/p}
\end{equation}
for any non-negative function $f\in L^p(\partial \Rset_+^n)$ with $q <0$ satisfies
\begin{equation}\label{eq:cond1}
\frac1q =  \frac{n-1}n\, \left(\frac1p -1\right ) - \frac\lambda n 
\end{equation}
which can be computed in terms of $r$ as follows: $1/q=1-1/r$. As a convention, for any $\vartheta<1$ any any function $\phi : \Omega \to \Rset$, by the notation $\phi \in L^\vartheta (\Omega)$, we mean $\int_\Omega |\phi|^\vartheta < +\infty$ although this integral is no longer a norm for $L^\vartheta (\Omega)$ with $\vartheta<1$ since the triangle inequality fails to hold. Being a linear topological space, it is well-known that $L^\vartheta (\Omega)$ with $\vartheta \in (0,1)$ has trivial dual.

Similar, one can consider the ``restriction" operator $R_\lambda $ which maps any function $g$ on $\Rset_+^n$ to a function on $\partial \Rset_+^n$ as the following
\[
(R_\lambda g)(x) = \int_{\Rset_+^n} |y-x|^\lambda g(y) dy
\]
for $x\in \partial \Rset_+^n$. Note that the operators $I_\lambda$ and $R_\lambda$ are dual in the sense that for any functions $f$ on $\partial \Rset_+^n$ and $g$ on $\Rset_+^n$, the following identity
\[
\int_{\Rset_+^n} (I_\lambda f)(y) g(y) dy = \int_{\partial \Rset_+^n} f(x) (R_\lambda g)(x) dx
\]
holds, thanks to the Tonelli theorem.

Once we introduce $R_\lambda $, we can easily see that the reversed HLS inequality \eqref{eq:reverseHLS} is equivalent to the following inequality
\begin{equation}\label{eq:anotherform1}
\Big( \int_{\partial \Rset_+^n} |R_\lambda g|^q dx \Big)^{1/q} \geqslant \Cscr_{n,\alpha,p}^+ \Big( \int_{ \Rset_+^n} |g|^r dx \Big)^{1/r} 
\end{equation}
for any non-negative function $g\in L^r( \Rset_+^n)$ with $q<0$ satisfies
\begin{equation}\label{eq:cond2}
\frac{1}{q}  = \frac{n}{{n - 1}}\left( {\frac{1}{r} - 1} \right) - \frac{\lambda }{n - 1}
\end{equation}
which can be computed in terms of $p$ as follows: $1/q=1-1/p$. 

In view of \eqref{eq:anotherform}, to study the existence of extremal functions for \eqref{eq:RHLS}, it is equivalent to studying the following minimizing problem
\begin{equation}\label{eq:variationalprob}
\Cscr_{n,\alpha,p}^+ :=\inf_f \Big\{ \|E_\lambda f\|_{L^q(\Rset_+^n)} : f\geqslant 0, \|f\|_{L^p(\partial \Rset_+^n)} =1\Big\}.
\end{equation}
It is elementary to verify that extremal functions for \eqref{eq:reverseHLS} are those solving the minimizing problem \eqref{eq:variationalprob}. In the following result, we prove that such a extremal function for \eqref{eq:reverseHLS} indeed exists.

\begin{theorem}\label{Existence}
There exists some function $f\in L^p(\partial \Rset_+^n)$ such that $f \geqslant 0$, $\|f\|_{L^p(\partial \Rset_+^n)} =1$ and $\|E_\lambda f\|_{L^q(\Rset_+^n)} = \Cscr_{n,\alpha,p}^+$. Moreover, if $f$ is a minimizer of \eqref{eq:variationalprob} then there exist a non-negative, strictly decreasing function $h$ on $[0, +\infty)$ and some $x_0\in \partial \Rset_+^n$ such that $f(x) = h(|x+x_0|)$ a.e. $x\in \partial \Rset_+^n$.
\end{theorem}

To prove the existence of extremal functions for \eqref{eq:reverseHLS}, we borrow Talenti's proof of the sharp Sobolev inequality by considering \eqref{eq:variationalprob} within the set of symmetric decreasing rearrangements. In view of \eqref{eq:variationalprob}, if we denote by $f^\star$ the symmetric decreasing rearrangement of some function $f \in L^p(\partial \Rset_+^n)$, then it is easy to see that $\|f^\star\|_{L^p(\partial \Rset_+^n)} = \|f\|_{L^p(\partial \Rset_+^n)} $. Therefore, it suffices to compare $\|E_\lambda f\|_{L^q(\Rset_+^n)}$ and $\|E_\lambda f^\star\|_{L^q(\Rset_+^n)} $. The key ingredient in our analysis is to characterize $E_\lambda f^\star$ by showing that it depends on two parameters: one is the distance from the boundary and the other is some radial variable on this boundary; see Lemma \ref{lemmaexistence}. Using the characterization of $E_\lambda f^\star$, we successfully obtain the existence result.

We now turn our attention to the extremal function found in Theorem \ref{Existence} above. In order to discuss further, let us first denote the following functional
\[
F_\lambda (f,g) = \int_{\Rset_+^n} \int_{\partial\Rset_+^n} f(x) |x-y|^\lambda g(y) dx dy
\]
for any nonnegative functions $f\in L^p(\partial \Rset_+^n)$ and $g\in L^r(\Rset_+^n)$. Then, in order to study the existence of extremal functions, it is necessary to minimize the functional $F_\lambda$ along with the following two constraints $\int_{\partial\Rset_+^n} |f(x)|^p dx =1$ and $\int_{\Rset_+^n} |g(y)|^r dy =1$. Upon a simple calculation, with respect to the function $f$, the first variation of the functional $F_\lambda$ is nothing but
\[
D_f (F_\lambda) (f,g) (h) = \int_{\partial\Rset_+^n} \Big( \int_{\Rset_+^n} |x-y|^\lambda g(y) dy \Big) h(x) dx
\]
while the first variation of the constraint $\int_{\partial\Rset_+^n} |f(x)|^p dx =1$ is as follows
\[
p \int_{\partial\Rset_+^n} |f(x)|^{p-2} f(x) h(x) dx.
\]
Therefore, by the Lagrange multiplier theorem, there exists some constant $\alpha$ such that
\[
\int_{\partial\Rset_+^n} \Big( \int_{\Rset_+^n} |x-y|^\lambda g(y) dy \Big) h(x) dx = \alpha \int_{\partial\Rset_+^n} |f(x)|^{p-2} f(x) h(x) dx
\]
holds for all $h$ defined in $\partial \Rset_+^n$. From this we know that $f$ and $g$ must satisfy the following equation
\[
\alpha |f(x)|^{p-2} f(x) = \int_{\Rset_+^n} |x-y|^\lambda g(y) dy.
\]
Interchanging the role of $f$ and $g$, we also know that $f$ and $g$ must fulfill the following
\[
\beta |g(y)|^{r-2} g(y) = \int_{\partial \Rset_+^n} |x-y|^\lambda f(x) dx
\]
for some new constant $\beta$. The balance condition guarantees that $\alpha = \beta = 1/F_\lambda (f,g)$. Hence, up to a constant multiple and simply using the following changes $u =f^{p-1}$ and $v = g^{r-1}$, the two relations above lead us to studying the following integral system
\begin{equation}\label{eq:systemDouZhu}
\left\{
\begin{split}
u(x) &= \int_{\Rset_+^n} |x-y|^\lambda v(y)^{1/(r-1)} dy,\\
v(y) &= \int_{\partial \Rset_+^n} |x-y|^\lambda u(x)^{1/(p-1)} dx.
\end{split}
\right.
\end{equation}
Note that the exponents $1/(r-1)$ and $1/(p-1)$ in \eqref{eq:systemDouZhu} are all negative. Concerning to the integral system \eqref{eq:systemDouZhu}, in the case when $1-n < \lambda <0$ and when $r=2n/(2n+\lambda)$ and $p=2(n-1)/(2n+\lambda-2)$, all non-negative integrable solutions of \eqref{eq:systemDouZhu} was already classified in \cite[Section 3]{dz2013} using an integral form of the well-known method of moving spheres. In the literature, the method of moving spheres was first introduced by Li and Zhu in \cite{lz1995}, see also \cite{l2004, xu2005}, which is a variant of the well-known method of moving planes introduced by Aleksandrov \cite{a1958}, see also \cite{s1971, gnn1979, cgs1989, cl1991, clo2005, clo2006}.

The main result in \cite[Section 3]{dz2013} is to show that, up to translations and dilations, any non-negative, measurable solution $(u,v)$ of \eqref{eq:systemDouZhu} must be the following form
\[\left\{
\begin{split}
u(x) =& a_1{\left( {|x - \overline x{|^2} + b^2} \right)^{(\alpha - n)/2}}, \\
v(x,0) = &a_2 {\left( {|x - \overline x{|^2} + b^2} \right)^{(\alpha - n)/2}},
\end{split}
\right.\]
where $a_1, a_2, b>0$ and $x, \overline x \in \partial \Rset_+^n$. Motivated by the above classification by Dou and Zhu, in the last part of our present paper, we also classify solutions of integral systems of the form \eqref{eq:systemDouZhu} where $\lambda>0$. To be precise, we are interested in classification of nonnegative, measurable functions of the following general system
\begin{equation}\label{eqIntegralSystem}
\left\{
\begin{split}
u(x) &= \int_{\Rset_+^n} |x-y|^\lambda v(y)^{-\kappa} dy,\\
v(y) &= \int_{\partial\Rset_+^n} |x-y|^\lambda u(x)^{-\theta} dx,
\end{split}
\right.
\end{equation}
with $\lambda, \kappa, \theta>0$. To achieve that goal, we first establish the following necessary condition.

\begin{lemma}\label{lemNECESSARY}
For $n \geqslant 3$, $\lambda>0$, $ \kappa>0$ and $\theta>0$, then a necessary condition for $\kappa$ and $\theta$ in order for \eqref{eqIntegralSystem} to admit a $C^1$ solution $(u,v)$ defined in $\partial\Rset_+^n \times \Rset_+^n$ is
\begin{equation}\label{eq:NecessaryCond}
\frac{{n - 1}}{n}\frac{1}{{\theta - 1}} + \frac{1}{{\kappa - 1}} = \frac{\lambda }{n}.
\end{equation}
\end{lemma}

The condition \eqref{eq:NecessaryCond} usually refers to the critical condition for \eqref{eqIntegralSystem}. Then, we provide the following classification result for solutions of \eqref{eqIntegralSystem} in the case $\kappa = \theta + 2/\lambda$.

\begin{proposition}\label{thmCLASSIFICATION}
Given $n \geqslant 3$, suppose that $\lambda>0$, $\kappa>0$ and $\theta>0$ satisfy $\kappa = \theta + 2/\lambda$. Let $(u,v)$ be a pair of nonnegative Lebesgue measurable functions defined in $\partial\Rset_+^n \times \Rset_+^n$ satisfying \eqref{eqIntegralSystem}. Then $\kappa = 1 + 2n/\lambda$ and hence $\theta = 1 + (2n - 2)/\lambda$ and for some constants $a, b>0$ and some point $\overline x \in \partial\Rset_+^n$, $u$ and $v$ take the following form
\[
u(x)=v(x,0) = a( |x-\overline x|^2 + b^2)^{\lambda/2}
\]
for $x \in \partial \Rset_+^n$.
\end{proposition}

To prove Proposition \ref{thmCLASSIFICATION}, we make use of the method of moving spheres introduced in \cite{lz1995}; see also \cite{lz2003}. For the classical HLS inequality, it is deserved to note that Frank and Lieb \cite{fl2010} successfully used reflection positivity via inversions in spheres to replace the moving spheres argument.  It is likelihood that the inversion positivity could be used in this new scenario and we hope we could treat this issue elsewhere.

As an immediate consequence of Proposition \ref{thmCLASSIFICATION}, we can explicitly compute the sharp constant in the reversed HLS inequality \eqref{eq:reverseHLS} for the special case when $\lambda =2$. Note that in this special case, there hold $p = (n-1)/n$ and $r = n/(n+1)$.
\begin{corollary}\label{explicit}
Let $n\geqslant 2$, then
\begin{equation}\label{eq:explicit}
\Cscr_{n, 1-1/n, n/(n+1)}^+ =\frac{2^{-1+1/n}}\pi \left(\frac{\Gamma(n)}{\Gamma(n/2)}\right)^{1/n} \left(\frac{\Gamma(n-1)}{\Gamma((n-1)/2)}\right)^{1/(n-1)}.
\end{equation}
\end{corollary}

Again, we note that Corollary \ref{explicit} only applies to the case $\lambda = 2$. For $0 < \lambda \ne 2$, it is not easy to obtain a precise value for the sharp constant $\Cscr_{n,\alpha,p}^+$. This is because we do not know much information of $v$ out of the boundary $\partial \Rset_+^n$. 
For general case, it has just come to our attention that the usage of the Gegenbauer polynomials could be useful and we will address this issue in future work; see \cite{fl2010, fl2011, fl2012a}.

Finally, we study the limiting case of \eqref{eq:reverseHLS} when $\lambda =0$ which will be called the log-HLS inequality on half space; see \cite{CarlenLoss92}.

\begin{theorem}\label{log-HLSonhalfspace}
Let $n\geqslant 2$. There exists a constant $C_n$ such that for any positive functions $f\in L^1(\partial\Rset_+^n)$ and $g\in L^1(\Rset_+^n)$ such that $\int_{\partial\Rset_+^n} f(x) dx = \int_{\Rset_+^n} g(y) dy =1$, and 
\[
\int_{\partial\Rset_+^n} f(x)\ln(1 +|x|^2) dx < +\infty,\quad\text{and}\quad \int_{\Rset_+^n} g(y) \ln(1 + |y|^2) dy < +\infty,
\]
then there holds
\begin{align}\label{eq:log-HLS}
-\int_{\Rset_+^n}& \int_{\partial\Rset_+^n} f(x) \ln(|x-y|) g(y) dx dy\notag\\
&\leqslant \frac1{2(n-1)}\int_{\partial\Rset_+^n} f(x) \ln f(x) dx + \frac1{2n}\int_{\Rset_+^n} g(y) \ln g(y) dy - C_n.
\end{align}
The constant $C_n$ is given by
\begin{align}\label{bestconstantlogHLS}
C_n = &-\frac1{2(n-1)} \ln |\mathbb S^{n-1}| + \frac1{2(n-1)|\mathbb S^{n-1}|}\int_{\partial\Rset_+^n}f_0(x) \ln f_0(x) dx\notag\\
& - \frac1{2n} \ln\Big(\int_{\Rset_+^n} \exp \Big(-\frac{2n}{|\mathbb S^{n-1}|}\int_{\partial\Rset_+^n} \ln(|x-y|) f_0(x)dx \Big) dy \Big),
\end{align}
with
\[
f_0(x) = \left(\frac2{1+|x|^2}\right)^{-n+1},\quad x \in \partial \Rset_+^n.
\]
Moreover, the inequality \eqref{eq:log-HLS} is sharp, and equality occurs if $f = f_0/|\mathbb S^{n-1}|$ and
\[
g(x) = c_n\exp\Big(-\frac{2n}{|\mathbb S^{n-1}|} \int_{\partial\Rset_+^n} f_0(x) \ln |x-y| dx\Big),
\]
where the constant $c_n$ is chosen in such a way that $\int_{\Rset^n_+} g(y) dy =1$.
\end{theorem}

In the forthcoming article \cite{NgoNguyen2015Heisenberg}, we shall study the reversed HLS inequality on the Heisenberg group $\mathbb H^n$.


\section{Proof of reversed HLS inequality in $\Rset_+^n$: Proof of Theorem \ref{thmMAIN}}

In this section, we prove the reversed HLS inequality \eqref{eq:reverseHLS}. By homogeneity, we can normalize $f$ and $g$ in such a way that $\|f\|_{L^p(\partial \Rset_+^n)} = \|g\|_{L^r(\Rset_+^n)} =1$. For each point $y\in \Rset^n$, let us denote
\[
B_c(x) = \{z\in \Rset^n\, :\, |z-y| \leqslant c\}.
\]
In the special case when $y = 0$, we simply denote $B_c(0)$ by $B_c$; hence $B_c = \{y\in \Rset^n\, :\, |y| \leqslant c\}$. For $a,b> 0$, we also denote 
\[
u(a) = \mathscr L^{n-1}(\{x \in \partial \Rset_+^n \, :\, f(x) > a\})
\]
and
\[
v(b) =\mathscr L^n(\{y \in \Rset_+^n \,:\, g(y) > b\}),
\]
where $\mathscr L^k$ stands for the $k$-dimensional Lebesgue measure with positive integers $k$. Then by the layer cake representation \cite[Theorem 1.13]{liebloss2001} and our normalization it follows that
$$p\int_0^\infty u(a) a^{p-1} da = \|f\|_{L^p(\partial \Rset_+^n)}^p =1$$
and
$$r\int_0^\infty v(b) b^{r-1} db = \|g\|_{L^r(\Rset_+^n)}^r =1.$$
Next we denote $\lambda =\alpha -n >0$ and
\[
I(f,g) = \int_{\Rset_+^n} \int_{\partial \Rset_+^n} f(x) |x-y|^{\alpha-n} g(y) dx dy .
\]
As the first step, we establish a rough form for \eqref{eq:reverseHLS} by showing that there is some constant $C>0$ depending only on $n,p, \lambda$ such that $I(f,g) \geqslant C$. To this purpose, by applying the layer cake representation again, we obtain
\[
f(x) = \int_0^\infty \chi_{\{f> a\}}(x) da,\quad g(y) = \int_0^\infty \chi_{\{g> b\}}(y) db,
\]
and
$$|x-y|^{\lambda} = \lambda \int_0^\infty c^{\lambda-1} \chi_{\R^n\setminus B_c}(x-y) dc.$$ 
From this and the Fubini theorem, it follows that 
\begin{equation}\label{eq:layercake}
I(f,g) = \lambda \int_0^\infty\int_0^\infty\int_0^\infty c^{\lambda-1} I(a,b,c) da db dc,
\end{equation}
where
$$I(a,b,c) = \int_{\Rset_+^n} \int_{\partial \Rset_+^n} \chi_{\{f > a\}}(x) \chi_{\R^n\setminus B_c}(x-y) \chi_{\{g> b\}}(y) dx dy.$$

\noindent\textbf{Step 1}. Our first step to prove \eqref{eq:reverseHLS} is to claim the following: There holds
\begin{equation}\label{eq:Step1}
I(a,b,c) \geqslant \frac{u(a)\, v(b)}2
\end{equation}
for any $c$ satisfying
\begin{equation}\label{eq:claim}
c\leqslant \max\Big\{\Big(\frac{u(a)}{2\omega_{n-1}}\Big)^{1/(n-1)},\Big(\frac{v(b)}{\omega_n}\Big)^{1/n}\Big\} 
\end{equation}
where $\omega_k$ denotes the volume of unit ball in the $k$-dimensional space $\Rset^k$, that is, $\omega_k = \text{vol}(\mathbb B^k)$. Indeed, there are two possible cases regarding to the right hand side of \eqref{eq:claim}.

\noindent\underline{Case 1.1}. Suppose that $( u(a)/(2\omega_{n-1}) )^{1/(n-1)} \leqslant (v(b)/\omega_n)^{1/n}$. Then by the Fubini theorem, we can estimate $I(a,b,c)$ as follows
\begin{align*}
I(a,b,c)& = \int_{\partial \Rset_+^n} \chi_{\{f> a\}}(x) \mathscr L^n \big(\{g> b\} \cap \{y\in \Rset_+^n\,:\, |x-y| > c\} \big) dx\\
& = \int_{\partial \Rset_+^n} \chi_{\{f> a\}}(x) \Big(v(b) - \mathscr L^n \big(\{g> b\} \cap \{y\in \Rset_+^n\,:\, |x-y| \leqslant c\} \big) \Big) dx\\
&\geqslant \int_{\partial \Rset_+^n} \chi_{\{f> a\}}(x)\Big(v(b) -\frac{\omega_n\, c^n}2 \Big) dx\\
&\geqslant \frac{u(a) \, v(b)}2,
\end{align*}
which implies \eqref{eq:Step1}.

\noindent\underline{Case 1.2}. Otherwise, we suppose that $( v(b)/\omega_n)^{1/n} \leqslant (u(a)/(2\omega_{n-1}))^{1/(n-1)}$. In this scenario, by repeating the same arguments as above, we can also bound $I(a,b,c)$ from below as $I(a,b,c) \geqslant u(a)v(b)/2$. Hence we conclude that \eqref{eq:Step1} holds.

\noindent\textbf{Step 2}. Using \eqref{eq:layercake}, \eqref{eq:claim} and the nonnegativity of function $I(a,b,c)$, we get
\begin{equation}\label{eq:step1}
\begin{split}
I(f,g) &\geqslant \int_0^\infty\int_0^\infty \Big(\lambda \int_0^{\max\left\{ ( u(a) / (2\omega_{n-1}) )^{1/(n-1)}, ( v(b)/\omega_n )^{1/n}\right\}}c^{\lambda-1}I(a,b,c) dc\Big)da\, db \\
&\geqslant \int_0^\infty\int_0^\infty \frac{u(a)\, v(b)}2\, \Big(\max\Big\{\left(\frac{u(a)}{2\omega_{n-1}}\right)^{1/(n-1)},\left(\frac{v(b)}{\omega_n}\Big)^{1/n}\right\}\Big)^\lambda\, da\, db \\
&\geqslant C \int_0^\infty\int_0^\infty u(a)\, v(b) \, \max\left\{u(a)^{\lambda/(n-1)}, v(b)^{\lambda /n}\right\}\, da\, db,
\end{split} 
\end{equation}
where $C> 0$ depends only on $n$. In the sequel, we use $C$ to denote a positive constant which depends only on $n,p, \lambda$ (or equivalently, on $n,p,\alpha$) and whose value can be changed from line to line. 

Denote $\beta =np/((n-1)r)$. We split the integral $\int_0^\infty$ evaluated with respect to the variable $b$ into two integrals as follows $\int_0^\infty = \int_0^{a^\beta} + \int_{a^\beta}^\infty$. Thus, the integrals in \eqref{eq:step1} can be estimated from below as the the following
\begin{equation}\label{eq:step2}
\begin{split}
\int_0^\infty\int_0^\infty & u(a)\, v(b) \, \max\left\{u(a)^{ \lambda/(n-1)}, v(b)^{\lambda /n}\right\}\, da\, db, \\
\geqslant &\int_0^\infty u(a) \int_0^{a^\beta } v(b)^{1+\lambda /n} db da + \int_0^\infty u(a)^{1+ \lambda/(n-1)} \int_{a^\beta }^\infty v(b) db da, \\
= &\int_0^\infty u(a) \int_0^{a^\beta } v(b)^{1+ \lambda /n} db da + \int_0^\infty v(b) \int_0^{b^{1/\beta }} u(a)^{1+ \lambda /(n-1)} da db\\
=&I+II.
\end{split}
\end{equation}

\noindent\textbf{Step 3}. We continue estimating the two integrals $I$ and $II$ in \eqref{eq:step2}. First, by using the reversed H\"older inequality, we get
\begin{equation}\label{eq:firstterm}
\begin{split}
\int_0^{a^\beta } v(b)^{1+ \lambda /n} db& = \int_0^{a^\beta } v(b)^{1+ \lambda /n} b^{(r-1)(1+ \lambda /n)}b^{-(r-1)(1+ \lambda /n)} db \\
&\geqslant \Big(\int_0^{a^\beta } v(b) b^{r-1} db\Big)^{1+\lambda /n} \Big(\int_0^{a^\beta } b^{(r-1)(1+ \lambda /n)\lambda /n} db\Big)^{-\lambda /n} \\
&= C a^{p-1} \Big(\int_0^{a^\beta } v(b) b^{r-1} db\Big)^{1+\lambda /n}.
\end{split}
\end{equation}
Similarly, we also get
\begin{equation}\label{eq:secondterm}
\int_0^{b^{1/\beta }} u(a)^{1+\lambda /(n-1)} da \geqslant C b^{r-1} \Big(\int_0^{b^{1/\beta }} u(a) a^{p-1} da \Big)^{1+\lambda /(n-1)}.
\end{equation}
Upon using our normalization $1 =r\int_0^\infty v(b) b^{r-1} dr$ and the fact that $1 + \lambda /n < 1 +\lambda /(n-1)$, we deduce 
\begin{equation}\label{eq:compare}
\Big(\int_0^{a^\beta } v(b) b^{r-1} db\Big)^{1+ \lambda /n} \geqslant C \Big(\int_0^{a^\beta } v(b) b^{r-1} db\Big)^{1+ \lambda /(n-1)}
\end{equation}
for any $a>0$. Setting $\gamma = 1 +\lambda /(n-1)$ and plugging \eqref{eq:firstterm}, \eqref{eq:secondterm}, and \eqref{eq:compare} into \eqref{eq:step2}, we eventually arrive at
\[
\begin{split}
\int_0^\infty\int_0^\infty u(a)\, v(b) \, \max &\left\{u(a)^{\lambda /(n-1)}, v(b)^{\lambda /n}\right\}\, da\, db \\
\geqslant &C \int_0^\infty u(a) a^{p-1}\Big(\int_0^{a^\beta } v(b) b^{r-1} db\Big)^{\gamma } da \\
&+ C\int_0^\infty v(b) b^{r-1} \Big(\int_0^{b^{1/\beta }} u(a) a^{p-1} da \Big)^{\gamma } db.
\end{split}
\]
Using the relation $r\int_0^\infty v(b) b^{r-1} dr = p\int_0^\infty u(a) a^{p-1} da =1$ and the convexity of the function $\Phi(t) = t^\gamma $ we obtain, thanks to the Jensen inequality, the following
\begin{equation}\label{eq:step3}
\begin{split}
\int_0^\infty\int_0^\infty u(a)\, v(b) \, \max &\left\{u(a)^{\lambda /(n-1)}, v(b)^{\lambda /n}\right\}\, da\, db\\
\geqslant C &\Big(\int_0^\infty u(a) a^{p-1}\int_0^{a^\beta } v(b) b^{r-1} db da\Big)^\gamma\\
&+C\Big(\int_0^\infty v(b) b^{r-1} \int_0^{b^{1/\beta }} u(a) a^{p-1} da db\Big)^\gamma .
\end{split}
\end{equation}
Also by the convexity of the function $\Phi(t) = t^\gamma $, we have the following elementary inequality $A^\gamma + B^\gamma \geqslant 2^{1-\gamma } (A+B)^\gamma$ for all $A,B>0$. By applying this elementary inequality to \eqref{eq:step3} and again using the Fubini theorem, we conclude that
\begin{align}\label{eq:finish}
\int_0^\infty &\int_0^\infty u(a)\, v(b) \, \max\left\{u(a)^{\lambda /(n-1)}, v(b)^{\lambda /n}\right\}\, da\, db,\notag\\
&\geqslant C\Big(\int_0^\infty u(a) a^{p-1}\int_0^{a^\beta } v(b) b^{r-1} db da + \int_0^\infty v(b) b^{r-1} \int_0^{b^{1/\beta }} u(a) a^{p-1} da db\Big)^\gamma \notag\\
&= C\Big(\int_0^\infty u(a) a^{p-1}\int_0^{a^\beta } v(b) b^{r-1} db da + \int_0^\infty u(a) a^{p-1}\int_{a^\beta }^\infty v(b) b^{r-1} db da\Big)^\gamma \notag\\
& = C \big(\frac1{pr}\big)^\gamma .
\end{align}
Combining \eqref{eq:step1} and \eqref{eq:finish} gives the estimate $I(f,g) \geqslant C$ for some constant $C>0$. From this, the sharp reversed Hardy-Littlewood-Sobolev inequality on the upper half space \eqref{eq:reverseHLS} follows where the sharp constant $\Cscr_{n,\alpha,p}^+$ is characterized by \eqref{eq:variationalprob}.


\section{Existence of extremal functions for reversed HLS inequality: Proof of Theorem \ref{Existence}}

Recall that $p,r \in (0,1)$ and $\lambda >0$ satisfy $(n-1)/np + 1/r - \lambda /n = 2 - 1/n$. As in Introduction and for the sake of simplicity, we still denote $q = r/(r-1) < 0$. Clearly, $1/p -\lambda /n = 1 + 1/q$. Let $f$ be a function on $\partial\Rset^n_+$ which vanishes at infinity, its symmetric decreasing rearrangement is denoted by $f^\star$; see \cite{liebloss2001} or \cite{Bur} for definition. It is well-known that if $f\in L^p(\partial\Rset^n_+)$ with $p > 0$, then $f^\star\in L^p(\partial\Rset^n_+)$ and $\|f\|_{L^p(\partial\Rset^n_+)} = \|f^\star\|_{L^p(\partial\Rset^n_+)}$.

We start the proof of Theorem \ref{Existence} by the following simple lemma which says more about the interaction between $f$ and $f^\star$.

\begin{lemma}\label{lemmaexistence}
We have the following claims:
\begin{itemize}
\item [(i)] There exists a positive function $F$ on $\Rset_+^2$ which is strictly increasing in each variable (when the other is fixed) such that 
$$(E_\lambda f^\star) (x',x_n) = F(|x'|,x_n)$$ 
for any $(x',x_n) \in \Rset^n_+$.
\item [(ii)] For any non-negative function $f\in L^p(\partial\Rset^n_+)$, there holds
\begin{equation}\label{eq:decreasenorm}
\int_{\Rset^n_+} |E_\lambda f|^q dy \leqslant \int_{\Rset^n_+} |E_\lambda f^\star |^q dy ,
\end{equation}
where $q = r/(r-1) < 0$ with the equality occurs if and only if $f^\star$ is a strictly decreasing and there exists $x_0\in \partial\Rset^n_+$ such that 
$$f(x) = f^\star(x +x_0)$$ 
a.e. $x\in \partial\Rset^n_+$.
\end{itemize}
\end{lemma}
\begin{proof}
The Lemma is immediately derived from the definition of the extension $E_\lambda$ and Lemma $1$ in \cite{NgoNguyen2015}.
\end{proof}

We are now in a position to prove Theorem \ref{Existence}. The radial symmetry and strictly decreasing of the minimizers for \eqref{eq:variationalprob} immediately follow from Lemma \ref{lemmaexistence}. We only have to prove the existence of a minimizer for this problem. For the sake of clarity, we divide our proof into several steps.

\noindent\textbf{Step 1}. \textit{Selecting a suitable minimizing sequence for \eqref{eq:variationalprob}.} 

We start our proof by letting $\{f_j\}_j$ be a minimizing sequence for the problem \eqref{eq:variationalprob}, so is the sequence $\{f_j^\star \}_j$. Hence, without loss of generality, we can assume at the very beginning that $\{f_j\}_j$ is nonnegative radially symmetric and non-increasing minimizing sequence. 

By abusing notations, we shall write $f_j(x)$ by $f_j(|x|)$ or even by $f_j (r)$ where $r=|x|$. Under this convention, by the normalization $\|f_j\|_{L^p(\partial \Rset_+^n)} =1$, we have
\[
\begin{split}
1 = & (n-1)\omega_{n-1} \int_0^\infty f_j(r)^p r^{n-2} dr \geqslant  \omega_{n-1} f_j(R)^p R^{n-1}
\end{split}
\]
for any $R > 0$. From this, we obtain the following estimate $0\leqslant f_j(r) \leqslant C r^{-(n-1)/p}$ for any $r > 0$ and for some constant $C$ independent of $j$. 

In order to go further, we need the following lemma whose proof mimics that of \cite[Lemma 3.2]{dz2014}; see also \cite[Lemma 2.4]{l1983}.

\begin{lemma}\label{farawayzero}
Suppose that $f\in L^p(\Rset^n)$ is non-negative, radially symmetric, and 
$$f(|y|) \leqslant \epsilon |y|^{-(n-1)/p}$$ 
for all $y\in \partial\Rset_+^n$. Then for any $p_1\in (0, p)$, there exists a constant $C_1 > 0$, independent of $f$ and $\epsilon$ such that 
\begin{equation}\label{eq:farawayzero}
 \int_{\Rset_+^n} |E_\lambda f|^q dy \leqslant C_1 \epsilon^{q(1 - p)/p_1}  \Big( \int_{\partial\Rset_+^n} | f|^p dx \Big)^{q/p_1}
\end{equation}
where $q = r/(r-1) < 0$.
\end{lemma}

\begin{proof}
Define $F: \Rset\to \Rset$ by setting
\[
F(t) = e^{(n-1)t/p} f(e^t).
\]
Then we can easily see that 
\begin{equation}\label{eq:relationfF}
\int_{\partial\Rset^n_+} |f|^p dx = (n-1)\omega_{n-1} \int_{-\infty} ^{+\infty} |F|^p dt
\end{equation}
and that
\begin{equation}\label{eq:relationfF2}
\|F\|_{L^\infty(\Rset)} \leqslant \epsilon,
\end{equation}
where, as before, $\omega_k = 2\pi^{k/2} / k \Gamma(k/2)$ is the volume of the $k$-dimensional unit ball $\mathbb B^k$. Writing $y =(y',y_n) \in \Rset^n_+$, it is easy to see that $E_\lambda f$ is radially symmetric in $y'$. Now we define $H: \Rset\times\Rset_+\to \Rset$ by letting
\[
H(t,y_n) = e^{nt/q} (E_\lambda f)(e^t, e^t y_n).
\]
By a simple change of variables, we then obtain
\begin{equation}\label{eq:relationEf}
(n-1)\omega_{n-1} \int_{\Rset_+^2} |H(t,y_n)|^q dtdy_n = \int_{\Rset_+^n} |E_\lambda f|^q dy .
\end{equation}
Thanks to \eqref{eq:condHLS} and recall $1/q = 1-1/r$ from the beginning of this section, we know that $n/q + \lambda /2 = -(n-1) (1-1/ p) - \lambda /2$. If for each real number $s \geqslant 0$ we use $\overrightarrow s$ to denote some vector sitting in $\partial \Rset_+^n$ with length $s$, then we clearly have
\[\begin{split}
H(t,y_n) = & e^{nt/q} \int_{\partial \Rset_+^n} \Big| | \overrightarrow{e^t}-x|^2+e^{2t}y_n^2 \Big|^{\lambda/2} f(x) dx \\
= & e^{nt/q + t\lambda/2} \int_{\partial \Rset_+^n} \big|  e^t (1+y_n^2) + e^{-t} |x|^2 -2 \overrightarrow{1} \cdot x  \big|^{\lambda/2} f(x) dx \\
= & e^{(n/q+ \lambda/2)t}  \int_{-\infty} ^{+\infty}  \int_{\partial B^{n-2}(0, e^s)} \big| e^{t-s} (1+y_n^2) + e^{-(t-s)}-2 \overrightarrow{e^{-s}} \cdot x \big|^{\lambda/2} \times \\
&\qquad \qquad \qquad \qquad  \qquad  \qquad \times e^{s(1+\lambda /2)} f(e^s) d\sigma ds \\
= & e^{(n/q+ \lambda/2)t}  \int_{-\infty} ^{+\infty}  \int_{\mathbb S^{n-2}} \big| e^{t-s} (1+y_n^2) + e^{-(t-s)}-2 \overrightarrow{1} \cdot \xi \big|^{\lambda/2} \times \\
&\qquad \qquad \qquad \qquad  \qquad  \qquad \times e^{s(n-1+\lambda /2)} f(e^s) d\xi ds  .
\end{split}\]
Hence, thanks to $n/q + \lambda  + n - 1 = (n - 1)/p$, we can readily obtain an explicit form of $H$ as follows
\[
H(t,y_n) = \int_{-\infty} ^{+\infty} L(t-s,y_n) F(s) ds,
\]
where $L(s,y_n) = e^{(n/q + \lambda/2)s} Z(s,y_n)$ with
\begin{equation*}
Z(s,y_n) = 
\begin{cases}
 \displaystyle \int_{\mathbb S^{n-2}} \big( e^s(1+y_n^2) + e^{-s} - 2 \overrightarrow{1} \cdot \xi \big)^{\lambda/2} d\xi & \text{ if } n\geqslant 3,\\
 \big( \big( e^s(1+y_n^2) + e^{-s} -2\big)^{\lambda/2} + \big( e^s(1+y_n^2) + e^{-s} +2\big)^{\lambda/2} \big) & \text{ if } n=2.
\end{cases}
\end{equation*}
Clearly, there exists a constant $c > 0$ such that $L(t,y_n) \geqslant c$ for all $(t,y_n) \in \Rset_+^2$ and 
\[
L(t,y_n) \sim (e^t(1+y_n^2) +e^{-t})^{\lambda/2} 
\]
as $t^2 +y_n^2 \to +\infty$. Since $n/q + \lambda = (n-1) (1/p - 1) > 0$, we know that $\lambda q + 1 <0$. From this we conclude, for any $s < 0$, that
\begin{equation}\label{eq:Step1-BoundForAllNegativeS}
\int_{\Rset} \left(\int_0^\infty L^q(t,y_n) dy_n\right)^{s/q} dt < + \infty.
\end{equation}
For any $p_1 \in (0, p)$, we choose $s_1$ such that $1/p_1 + 1/s_1 = 1 + 1/q$. Since $p_1 \in (0,1)$ and $q<0$, we clearly have $s_1 < 0$. In addition, it follows from $s_1/p_1 < s_1$ that $q/s_1 > 1$. From these facts and by the reversed Young inequality \cite[Lemma 2.2]{dz2014}, for any $y_n > 0$ we have 
\[
\int_{\Rset } |H(t, y_n)|^q dt \leqslant \Big( \int_{\Rset } |L(t, y_n)|^{s_1} dt \Big)^{q/s_1} \Big(\int_{\Rset } |F|^{p_1} dt\Big)^{q/p_1}.
\]
which, by integrating both sides with respect to $y_n$ over $[0, +\infty)$, implies
\begin{equation}\label{eq:rYoungineq}
 \int_{\Rset^2_+} |H(t, y_n)|^q dt dy_n \leqslant \Big(\int_{\Rset } |F|^{p_1} dt\Big)^{q/p_1} \int_0^\infty\Big(\int_{\Rset} L(t,y_n)^{s_1} dt\Big)^{q/s_1} dy_n.
\end{equation}
To estimate the right hand side of \eqref{eq:rYoungineq}, on one hand, we observe by the Minkowski inequality that
\begin{equation}\label{eq:Minkowskiineq}
\Big(\int_0^\infty\Big(\int_{\Rset} L^{s_1}(t,y_n) dt\Big)^{q/s_1} dy_n\Big)^{s_1/q} \leqslant \int_{\Rset} \left(\int_0^\infty L^q(t,y_n) dy_n\right)^{s_1/q} dt.
\end{equation}
Thanks to $q/s_1 > 0$, we conclude from \eqref{eq:Step1-BoundForAllNegativeS} and \eqref{eq:Minkowskiineq} that 
\[
\int_0^\infty\Big(\int_{\Rset} L^{s_1}(t,y_n) dt\Big)^{q/s_1} dy_n < +\infty,
\]
which then helps us to conclude from \eqref{eq:rYoungineq} that
\begin{equation}\label{eq:Step1-OneHand}
\int_{\Rset^2_+} |H(t, y_n)|^q dt dy_n \leqslant C \Big(\int_{\Rset } |F|^{p_1} dt\Big)^{q/p_1} .
\end{equation}
On the other hand, if we write $F ^{p_1} = F ^p F ^{p_1-p} \geqslant F ^p \|F\|_{L^\infty(\Rset)} ^{p_1-p}$, then thanks to \eqref{eq:relationfF2} and \eqref{eq:relationfF} we get
\begin{equation}\label{eq:Step1-OtherHand}
\int_{\Rset} |F|^{p_1} dt \geqslant C \epsilon^{p_1-p} \int_{\partial\Rset^n_+} |f|^p dx .
\end{equation}
Simply plugging \eqref{eq:Step1-OtherHand} into \eqref{eq:Step1-OneHand} gives
\begin{equation}\label{eq:Step1-BothHands}
\int_{\Rset^2_+} |H(t, y_n)|^q dt dy_n \leqslant C\epsilon^{q(1-p)/p_1} \Big( \int_{\partial\Rset^n_+} |f|^p dx \Big)^{q/p_1} .
\end{equation}
Thus, combining \eqref{eq:Step1-BothHands} and \eqref{eq:relationEf} gives \eqref{eq:farawayzero} as claimed.
\end{proof} 

\noindent\textbf{Step 2}. \textit{Existence of a potential minimizer $f_0$ for \eqref{eq:variationalprob}.} 

For each $j$ we set
\[
a_j = \sup_{r > 0} r^{(n-1)/p} f_j(r) 
\]
which obviously belongs to $[0,C]$. Thanks to the normalization $\|f_j\|_{L^p(\partial\Rset^n_+)} = 1$ and the fact $\|E_\lambda f_j\|_{L^q(\Rset_+^n)} \to \Cscr_{n,p,\lambda}^+ < +\infty$, we obtain from Lemma \ref{farawayzero} the following estimate $a_j \geqslant 2c_0$ for some $c_0 >0$. For each $j$, we choose $\lambda_j > 0$ in such a way that $\lambda_j^{(n-1)/p} f_j(\lambda_j) > c_0$. Then we set 
\[
g_j(x) = \lambda_j^{n/p} f_j(\lambda_j x).
\] 
From this, it is routine to check that $\{g_j\}_j$ is also a minimizing sequence for problem \eqref{eq:variationalprob}, and $g_j(1) > c_0$ for any $j$ by our choice for $\lambda_j$. Consequently, by replacing the sequence $\{f_j\}_j$ by the new sequence $\{g_j\}$, if necessary, we can further assume that our sequence $\{f_j\}_j$ obeys $f_j(1) > c_0$ for any $j$. 

Similar to Lieb's argument in \cite{l1983}, which is based on the Helly theorem, by passing to a subsequence, we have $f_j\to f_0$ a.e. in $\partial\Rset_+^n$. It is now evident that $f_0$ is non-negative radially symmetric, non-increasing and is in $L^p(\partial\Rset_+^n)$. The rest of our arguments is to show that $f_0$ is indeed the desired minimizer for \eqref{eq:variationalprob}.

By Lemma \ref{lemmaexistence}, we know that $(E_\lambda f_j)(y',y_n)$ is radially symmetric and strictly decreasing in $y'$ and strictly increasing in $y_n$ for any $j$. Moreover, for all $y \in \Rset_+^n$, there holds
\begin{equation}\label{eq:lowerbound}
(E_\lambda f_j) (y) \geqslant c_0\int_{|x| \leqslant 1} |x-y|^\lambda dx \geqslant C_2 (1 + |y|^\lambda)=: g(y)
\end{equation}
for some new constant $C_2$ independent of $j$. 

\noindent\textbf{Step 3}. \textit{The function $f_0$ is indeed a minimizer for \eqref{eq:variationalprob}} 

For each $y\in \Rset_+^n$, set
\[
k(y) = \liminf_{j\to\infty} (E_\lambda f_j) (y) .
\]
By \eqref{eq:lowerbound}, we have $k(y) \geqslant g(y)$ for any $y\in \Rset^n_+$. It is easy to see that
\[
g(y)^q - k(y)^q = \liminf_{j\to\infty}(g(y)^q - (E_\lambda f_j) (y)^q).
\]
Again by \eqref{eq:lowerbound} and the Fatou lemma, we have
\[\begin{split}
\int_{\Rset_+^n} (g(y)^q -k(y)^q) dy \leqslant & \liminf_{j\to\infty} \int_{\Rset_+^n} (g(y)^q - (E_\lambda f_j) (y)^q) dy \\
= &\int_{\Rset_+^n} g(y)^q dy - \big( \Cscr_{n,p,\lambda}^+ \big) ^q.
\end{split}\]
Therefore
\begin{equation}\label{eqIntegralK=C}
\int_{\Rset_+^n} g(y)^q dy \geqslant \int_{\Rset_+^n} k(y)^q dy \geqslant \big( \Cscr_{n,p,\lambda}^+ \big) ^q.
\end{equation}
These inequalities imply that the set $\{y\in \Rset_+^n\, :\, 0 < k(y) < +\infty\}$ has positive measure. Hence we can take a point $y_1\in \Rset_+^n$ and extract a subsequence of $E_\lambda f_j$, still denoted by $E_\lambda f_j$, such that 
\[
\lim_{j\to +\infty} (E_\lambda f_j) (y_1) = a_1 \in (0, +\infty).
\]
Repeating the above arguments and extracting a subsequence of $E_\lambda f_j$ if necessary, we can choose a point $y_2 \in \Rset_+^n$ such that $y_2\not = y_1$ and that
\[
\lim_{j\to +\infty} (E_\lambda f_j) (y_2) = a_2 \in (0, +\infty).
\]
Then there exists some constant $C_5 > 0$ such that $(E_\lambda f_j) (y_i) \leqslant C_5$ for $i=1,2$ and for all $j \geqslant 1$. Using the simple inequality $|a+b|^\lambda \leqslant \max\{1,2^{\lambda-1}\} (|a|^\lambda + |b|^\lambda)$ for any $a,b\in\Rset^n$, we have
\begin{equation*}
\begin{split}
|y_1-y_2|^\lambda \int_{\partial\Rset_+^n}f_j(x)dx \leqslant &\max\{1,2^{\lambda-1}\} \int_{\partial\Rset_+^n} |y_1-x|^\lambda f_j(x) dx \\
&+ \max\{1,2^{\lambda-1}\} \int_{\partial\Rset_+^n} |y_2-x|^\lambda f_j(x) dx \\
=&\max\{1,2^{\lambda-1}\} \big( (E_\lambda f_j) (y_1) + (E_\lambda f_j) (y_2) \big)\\
\leqslant &2\max\{1,2^{\lambda-1}\} C_5.
\end{split}
\end{equation*}
Thus, there exists another constant $C_6> 0$ such that $\int_{\partial\Rset_+^n}f_j(x)dx \leqslant C_6$ for all $j\geqslant 1$. On one hand, for any $R > 2|y_1|$, there holds $|y_1-x| \geqslant |x|/3$ for any $x$ in the region $\{3R/4 \leqslant |x|\leqslant R\}$. Therefore, by a simple variable change, we can estimate
\[
\begin{split}
C_5 \geqslant &\int_{\{3R/4\leqslant |y|\leqslant R\}} |y_1-x|^\lambda f_j(x) dx \\
\geqslant & 3^{-\lambda} f_j(R) R^{n-1+\lambda} \int_{\{3/4\leqslant |x|\leqslant 1\}} |x|^\lambda dx.
\end{split}
\]
Note that in the preceding estimate, we have used the fact that $f_j$ is radial symmetric and non-increasing. Hence, there exists some $C_7 > 0$ such that $f_j(r) \leqslant C_7 r^{-n-\lambda+1}$ for any $r > 2 |y_1|$ and for all $j\geqslant 1$.

Making use of the above estimate $f_j(r) \leqslant C_7 r^{-n-\lambda+1}$, for any $r > 2|y_1|$, we further have
\begin{equation}\label{eq:outsideball}
\begin{split}
\int_{\{|x| > R\}} f_j(x)^p dx \leqslant & C_7^p \int_{\{|x| > R\}}|x|^{-p(n+\lambda-1)}dx \\
=& -\frac{(n-1)\omega_{n-1} q}{np} C_7^pR^{np/q}.
\end{split}
\end{equation}
Thanks to $\int_{\Rset^n}f_j(x)dx \leqslant C_6$, we also have
\begin{equation}\label{eq:fjgeqR}
\int_{\{f_j > R\}} f_j(x)^p dx \leqslant R^{p-1} \int_{\partial\Rset_+^n} f_j(x) dx \leqslant C_6 R^{p-1}.
\end{equation}
In view of \eqref{eq:outsideball} and \eqref{eq:fjgeqR}, for arbitrary $\epsilon >0$, we can select $R > 2|y_1|$ sufficiently large in such a way that 
\[
\int_{\{|x| > R\}} f_j(x)^p dx < \frac\epsilon 2
\]
and that
\[
\int_{\{f_j > R\}} f_j(x)^p dx < \frac\epsilon 2.
\]
We now set $g_j(x) = \min\{f_j(x),R\}$ for each $j\geqslant 1$. Since $\int_{\partial\Rset_+^n} f_j(x)^p dx =1$, we have
\begin{equation*}
\begin{split}
\int_{\{|x|\leqslant R\}}g_j(x)^p dx& \geqslant \int_{\{|x|\leqslant R\}\cap \{f_j \leqslant R\}} f_j(x)^p dx\\
&=1 - \int_{\{|x|\leqslant R\}\cap \{f_j > R\}} f_j(x)^p dx -\int_{\{|x| > R\}} f_j(x)^p dx \geqslant 1-\epsilon.
\end{split}
\end{equation*}
For each $R$ fixed, the dominated convergence theorem guarantees that
\begin{equation*}
\lim_{j\to\infty} \int_{\{|x|\leqslant R\}}g_j(x)^p dx = \int_{\{|x|\leqslant R\}}\left(\min\{f_0(x),R\}\right)^p dx.
\end{equation*}
Therefore, by sending $R \to +\infty$, we arrive at
\[
\int_{\partial\Rset_+^n}f_0(x)^p dx \geqslant 1-\epsilon,
\]
for any $\epsilon >0$. From this we can conclude $\int_{\partial\Rset_+^n}f_0(x)^p dx \geqslant 1$. On the other hand, by the Fatou lemma, we have $\int_{\partial\Rset_+^n}f_0(x)^p dx \leqslant 1$. Therefore, we have just proved $\|f_0\|_{L^p(\partial\Rset_+^n)} =1$.

In the last part of the step, to realize that $f_0$ is indeed a minimizer for \eqref{eq:variationalprob}, we apply the Fatou lemma once again to get
\[
k(x) = \liminf_{j\to\infty} E_\lambda f_j(x) \geqslant E_\lambda f_0(x),
\]
for a.e. $x$ in $\Rset_+^n$. Hence, combining the preceding estimate and \eqref{eqIntegralK=C} gives
\[
\Cscr_{n,p,\lambda}^+ = \Cscr_{n,p,\lambda}^+ \|f_0\|_{L^p(\partial\Rset_+^n)} \leqslant \|E_\lambda f_0\|_{L^q(\Rset_+^n)} \leqslant \Big(\int_{\Rset_+^n} k(y)^{q} dy \Big)^{1/q} \leqslant \Cscr_{n,p,\lambda}^+.
\] 
This shows that $f_0$ is a minimizer for \eqref{eq:variationalprob}; hence finishing the proof.


\section{Basic properties of (\ref{eqIntegralSystem}) and Proof of Lemma \ref{lemNECESSARY}}

In this section, we first establish some basic properties of the system \eqref{eqIntegralSystem} which shall be needed to prove Lemma \ref{lemNECESSARY} and Proposition \ref{thmCLASSIFICATION}.

\subsection{Preliminaries}

In this subsection, we setup some preliminaries necessarily for our analysis. Here and in what follows, by $\lesssim$ and $\gtrsim$ we mean inequalities up to $p$, $q$, and dimensional constants. For the sake of simplicity, we denote $B_{\partial \Rset_+^n} (x,R) = \{ \xi \in \partial \Rset_+^n : |\xi - x| \leqslant R\}$ and $B_{ \Rset_+^n} (x,R) = \{ \eta \in \Rset_+^n : |\eta - x| \leqslant R\}$. We also denote 
\[
\Sigma_{x,R}^{n-1} = \Rset_+^n \backslash \overline{B_{\Rset_+^n} (x, R)} , \quad \Sigma_{x,R}^n= \partial \Rset_+^n \backslash \overline{B_{\partial \Rset_+^n} (x,R)} .
\]
We now establish the most important part of this section known as a prior estimates for solutions of \eqref{eqIntegralSystem} as stated in Lemma \ref{lem-Growth} below.

\begin{lemma}\label{lem-Growth}
For $n \geqslant 1$ and $\lambda,\kappa,\theta>0$, let $(u,v)$ be a pair of non-negative Lebesgue measurable functions in $\partial \Rset_+^n \times \Rset_+^n$ satisfying \eqref{eqIntegralSystem}. Then there hold
\begin{equation}\label{eqGrowth1}
\int_{\partial \Rset_+^n} {(1 + |x|^\lambda )u(x)^{ - \theta}dx} < +\infty , \quad \int_{\Rset_+^n} {(1 + |y|^\lambda )v(y)^{ - \kappa}dy} < +\infty ,
\end{equation}
and
\begin{equation}\label{eqGrowth2}
\begin{split}
\mathop {\lim }\limits_{|x| \to +\infty } \frac{u(x)}{|x|^\lambda} =& {\int_{\Rset_+^n} {v(y)^{ - \kappa}dy} } , \quad \mathop {\lim }\limits_{|y| \to +\infty } \frac{{v(y)}}{{|y|^\lambda }} = {\int_{\partial\Rset_+^n} {u(x)^{ - \theta }dx} } ,
\end{split}
\end{equation}
and $u$ and $v$ are bounded from below in the following sense 
\begin{equation}\label{eqGrowth3}
 \frac{1 + |x|^\lambda }{C} \leqslant u(x) \leqslant C(1 + |x|^\lambda )
\end{equation}
for all $x \in \partial\Rset_+^n$ and
\begin{equation}\label{eqGrowth4}
 \frac{1 + |y|^\lambda }{C} \leqslant v(y) \leqslant C(1 + |y|^\lambda )
\end{equation}
for all $y \in \Rset_+^n$ for some constant $C \geqslant 1$.
\end{lemma}

\begin{proof}
To prove our lemma, we first observe from \eqref{eqIntegralSystem} that both $u$ and $v$ are strictly positive everywhere in their domains and are finite within a set of positive measure. Hence there exist some $R>1$ sufficiently large and some Lebesgue measurable set $E \subset \overline{\Rset_+^n}$ such that
\begin{equation}\label{eqSetE}
E \subset \{ z : u(z) < R, v(z) < R\} \cap B_{\overline{\Rset_+^n}} (0,R)
\end{equation}
with $\mathscr L^n (E) \geqslant 1/R$. Using this, we can easily bound $v$ from below as follows
\[
\begin{split}
v(y) \geqslant \int_E {|x - y|^\lambda u(x)^{ - q}dx} \geqslant & \frac{1}{R^q}\int_E {|x - y|^\lambda dx} =\frac{1}{R^q}\int_{E+y} {|x|^\lambda dx} 
\end{split}
\]
for any $y \in \Rset_+^n$. Now we choose $\varepsilon > 0$ small enough and then fix it in such a way that $\vol (B_{\Rset_+^n} (0,\varepsilon)) < \mathscr L^n (E) /2$. Then we can estimate
\[
\begin{split}
\int_{E+y} {|x|^\lambda dx} & \geqslant \int_{E+y \backslash B_{\Rset_+^n} (0,\varepsilon)} {|x|^\lambda dx} \\
& \geqslant \varepsilon^\lambda \int_{E+y \backslash B_{\Rset_+^n} (0,\varepsilon)} dy \\
& = \varepsilon^\lambda \big( \mathscr L^n (E+y) -\vol (B_{\Rset_+^n} (0,\varepsilon)) \big).
\end{split}
\]
From this, it is clear that $v$ is bounded from below by some positive constant. The same reason applied to $u$ shows that there exists some constant $C_0>0$ such that 
\begin{equation}\label{eqUVBoundedFromBelow}
u(x), v(y) > C_0 
\end{equation}
for all $x \in \partial\Rset^n$ and $y \in \Rset^n$.

\noindent\textbf{Proof of \eqref{eqGrowth3}}. To prove this, we first consider $|x| \geqslant 2R$ where $R$ is defined through \eqref{eqSetE}. Note that for every $y \in E \subset B_{\overline{\Rset_+^n}}(0,R)$, there holds $|x - y| \geqslant |x| - |y| \geqslant |x|/2$, thanks to $|x| \geqslant 2R$. Using this we can estimate
\[
v(y) \geqslant \frac{1}{R^q}\int_E {|x - y|^\lambda dx} \geqslant \frac{\text{vol}(E)}{(2R)^\lambda} |y|^\lambda
\]
for any $|y| \geqslant 2R$. A similar argument also shows $u(x) \geqslant \text{vol}(E) (2R)^{-\lambda}|x|^\lambda $ in the region $\{x : |x| \geqslant 2R\}$. Hence, it is easy to select a large constant $C>1$ in such a way that \eqref{eqGrowth3} holds in the region $\{ |x| \geqslant 2R\}$. Thanks to \eqref{eqUVBoundedFromBelow}, we can further decrease $C$, if necessary, to obtain the estimate \eqref{eqGrowth3} in the ball $\{x \in \overline{\Rset_+^n}: |x| \leqslant 2R\}$; hence the proof of \eqref{eqGrowth3} follows.

\noindent\textbf{Proof of \eqref{eqGrowth1}}. We only need to estimate $v$ since $u$ can be estimated similarly. To this purpose, we first show that $u^{-q} \in L^1(\partial \Rset_+^n)$. Clearly for some $\overline x$ satisfying $1 \leqslant |\overline x| \leqslant 2$, there holds
\[
\int_{\partial \Rset_+^n} {|\overline x - x|^\lambda u{(x)^{ - \theta}}dx} = v(\overline x ) \in (0, + \infty ).
\]
Observer that for any $x \in \partial \Rset_+^n \backslash B_{\partial \Rset_+^n} (0,4)$, there holds $|\overline x - x | \geqslant |x| -|\overline x | > 1$; hence
\[
\int_{\partial \Rset_+^n \backslash B_{\partial \Rset_+^n} (0,4) } u(x)^{ - \theta}dx < \int_{\Rset^n} {|\overline x - x|^\lambda u{(x)^{ - \theta }}dx} < +\infty.
\]
In the small ball $B_{\partial \Rset_+^n} (0,4)$, thanks to \eqref{eqGrowth3}, it is obvious to verify that
\[
\int_{B_{\partial \Rset_+^n} (0,4) } u(x)^{ - \theta}dx \lesssim \int_{B_{\partial \Rset_+^n} (0,4)} (1+|x|^\lambda))^{-\theta} dx < +\infty.
\]
Thus, we have just shown that $u^{-\theta} \in L^1(\partial \Rset_+^n)$. In view of \eqref{eqGrowth1}, it suffices to prove that
\begin{equation}\label{eqGrowth1-suffice}
\int_{\partial \Rset_+^n} |x|^\lambda u(x)^{ - \theta}dx < +\infty.
\end{equation}
To see this, we again observe that $|x| \leqslant 2|\overline x - x| $ for all $x \in \partial \Rset_+^n \backslash B_{\partial \Rset_+^n} (0,4) $. Therefore,
\[
\int_{\partial \Rset_+^n \backslash B_{\partial \Rset_+^n}(0,4) } |x|^\lambda u(x)^{ - q}dx \lesssim \int_{\partial \Rset_+^n \backslash B_{\partial \Rset_+^n} (0,4)} {|\overline x - x|^\lambda u{(x)^{ -\theta}}dx} < +\infty.
\]
In the small ball $B_{\partial \Rset_+^n} (0,4)$, it is obvious to see that
\[
\int_{ B_{\partial \Rset_+^n} (0,4) } |x|^\lambda u(x)^{ - \theta }dx \lesssim \int_{B_{\partial \Rset_+^n} (0,4)}u(x)^{-\theta} dx < +\infty,
\]
thanks to $u^{- \theta} \in L^1(\partial \Rset_+^n)$. From this, \eqref{eqGrowth1-suffice} follows, so does \eqref{eqGrowth1}. 

\noindent\textbf{Proof of \eqref{eqGrowth2}}. We only consider the limit $|y|^{-\lambda} v(y)$ as $|y| \to +\infty$ since the limit $|x|^{-\lambda} u(x)$ can be proved similarly. Indeed, using \eqref{eqIntegralSystem}, we first obtain
\begin{equation}\label{eqProofGrowth2-1}
\begin{split}
\mathop {\lim }\limits_{|y| \to +\infty } \frac{ v(y) }{|y|^\lambda } =& \mathop {\lim }\limits_{|y| \to +\infty } {\int_{\partial \Rset_+^n} {\frac{{|x - y|^\lambda }}{{|y|^\lambda}} u{(x)^{ - \theta}}dx} }.
\end{split}
\end{equation}
Observe that as $|y| \to +\infty$, $( |x - y|/|y| )^\lambda u(x)^{ - \theta} \to u(x)^{ - \theta}$ almost everywhere $y$ in $\Rset_+^n$. Hence we can apply the Lebesgue dominated convergence theorem to pass \eqref{eqProofGrowth2-1} to the limit to conclude \eqref{eqGrowth1} provided we can show that $|x - y|^\lambda |x|^{ - \theta}u(x)^{ - q}$ is bounded by some integrable function. To this end, we observe that $|x-y|^\lambda \lesssim |x|^\lambda +|y|^\lambda $; hence, if $|x| > 1$ then
\[
{\left( {\frac{|x - y|}{|y|}} \right)^\lambda}u(x)^{ - \theta } \lesssim (1+|x|^\lambda)u(x)^{ - \theta}.
\]
Our proof now follows by observing $ (1 + |x|^\lambda )u(x)^{ -\theta} \in L^1(\partial \Rset_+^n)$ by \eqref{eqGrowth1}.

\noindent\textbf{Proof of \eqref{eqGrowth4}}. To see this, we first observe from \eqref{eqGrowth2} that there exists some large number $k>1/R$ such that 
\[
\frac{u(x)}{|x|^\lambda} < 1 + \int_{\Rset_+^n} v(y)^{-\kappa} dy
\]
in $\partial \Rset_+^n \backslash B_{\partial \Rset_+^n} (0,kR)$. In the ball $B_{\partial \Rset_+^n} (0,kR)$, it is easy to estimate $|x-y|^\lambda \lesssim |x|^\lambda + |y|^\lambda$ which helps us to conclude that
\[
u(x) \lesssim (kR)^\lambda \int_{\Rset_+^n} (1+|y|^\lambda) v(y)^{-\kappa} dy
\]
in the ball $B_{\partial \Rset_+^n} (0,kR)$. From this and our estimate for $u$ outside $B_{\partial \Rset_+^n}(0,kR)$, we obtain the desired estimate in $\partial \Rset_+^n$. Our estimate for $v$ follows similarly; hence we obtain \eqref{eqGrowth4} as claimed.
\end{proof}

Inspired by \eqref{eqGrowth1}, we prove the following simple observation.

\begin{lemma}\label{lem-IntegralU=IntegralV}
There holds $u \in L^{1-\theta} (\partial\Rset_+^n)$ and $v \in L^{1-\kappa} (\Rset_+^n)$. Moreover,
\[
\int_{\partial\Rset_+^n} u^{1-\theta} (x) dx = \int_{\Rset_+^n} \int_{\partial\Rset_+^n} v(y)^{-\kappa} |x-y|^{\lambda } u^{-\theta}(x) dx dy = \int_{\Rset_+^n} v^{1-\kappa} (y) dy .
\]
\end{lemma}

\begin{proof}
To see $u \in L^{1-\theta} (\partial\Rset_+^n)$, we observe from \eqref{eqGrowth1} and \eqref{eqGrowth3} that
\[
\int_{\partial\Rset_+^n} u^{1-\theta} (x) dx  \lesssim \int_{\partial \Rset_+^n} {(1 + |x|^\lambda )u(x)^{ - \theta}dx} < +\infty.
\]
A similar argument shows $v \in L^{1-\kappa} (\Rset_+^n)$ as claimed. The way to obtain the desired relation is elementary since
\[
u^{1-\theta} (x) = u^{-\theta}(x) \int_{\Rset_+^n} v(y)^{-\kappa} |x-y|^{\lambda } dy
\]
and
\[
v^{1-\kappa} (y) = v(y)^{-\kappa} \int_{\partial\Rset_+^n} |x-y|^{\lambda } u^{-\theta}(x) dx .
\]
Integrating both sides over suitable domains gives the desired result.
\end{proof}

In the following result, we prove a regularity result similar to \cite[Lemma 5.2]{l2004} obtained by Li.

\begin{lemma}\label{lem-Regularity}
For $n \geqslant 1$ and $\lambda,\kappa,\theta>0$, let $(u,v)$ be a pair of non-negative Lebesgue measurable functions in $\partial \Rset_+^n \times \Rset_+^n$ satisfying \eqref{eqIntegralSystem}. Then $u$ and $v$ are smooth.
\end{lemma}

\begin{proof}
Our proof is similar to that of \cite[Lemma 5.2]{l2004}. Let $R>0$ be arbitrary, first we decompose $u$ and $v$ into the following way
\[\begin{split}
 u(x) =& \big(u_R^1 + u_R^2 \big) (x) = \Big(\int_{|y| \leqslant 2R} + \int_{|y| > 2R} \Big)|x - y|^\lambda v{(y)^{ - \kappa}}dy, \\
 v(y) =& \big(v_R^1 + v_R^2 \big) (y) = \Big(\int_{|x| \leqslant 2R} + \int_{|x| > 2R} \Big)|x - y|^\lambda u(x)^{ - \theta}dx. \\ 
\end{split} \]
Thanks to \eqref{eqGrowth1}, we immediately see that we can continuously differentiate $u_R^2$ and $v_R^2$ under the integral sign for any $x \in \partial \Rset_+^n$ satisfying $|x|<R$. Consequently, $u_R^2 \in C^\infty (B_{\partial \Rset_+^n}(0,R))$ and $v_R^2 \in C^\infty (B_{\Rset_+^n}(0,R))$. 

In view of \eqref{eqGrowth3} and \eqref{eqGrowth4}, we know that $u^{-\theta} \in L^\infty (B_{\partial \Rset_+^n} (0, 2R))$. This and the following elementary inequality $\big| |x-y|^\lambda - |z-y|^\lambda \big| \lesssim |x-z|^{\min\{\lambda, 1\}}$ for all $x, z \in B(0, R)$ and all $y \in B(0, 2R)$ conclude that $v_R^1$ is at least H\"older continuous in $B_{\Rset_+^n}(0,R)$. Similar reasons tell us that $u_R^1$ is also at least H\"older continuous in $B_{\partial \Rset_+^n}(0,R)$. Hence, we have just proved that $u$ and $v$ are at least H\"older continuous in $B(0,R)$, so are at least H\"older continuous in $\partial \Rset_+^n$ and $\Rset_+^n$ respectively since $R>0$ is arbitrary. 

Standard bootstrap argument shows $u \in C^\infty (\partial \Rset_+^n)$ and at the same time $v \in C^\infty (\Rset_+^n)$ follows the same lines.
\end{proof}

\subsection{Proof of Lemma \ref{lemNECESSARY}}

To prove this lemma, we borrow the idea in \cite{lei2015}. 
As the first step in the proof, we make use of the integrability of $u$ and $v$ in $L^{1-\theta} (\partial\Rset_+^n)$ and in $L^{1-\kappa} (\Rset_+^n)$ respectively to derive \eqref{eq-p5-eqKeyU} and \eqref{eq-p5-eqKeyV} below. For this purpose, we let $\zeta : [0,+\infty) \to [0,1]$ be a smooth cut-off function such that
\[
\zeta (t)=
\begin{cases}
0 &\text{ if } t \geqslant 2,\\
1 &\text{ if } 0 \leqslant t \leqslant1,\\
\end{cases}
\]
and $\zeta' \in [0,1]$. Then for some fixed $N>0$, by integration by parts, we obtain
\begin{equation}\label{eq-p5-eq3}
\begin{split}
\int_{B_{\partial\Rset_+^n}(0,N)} \zeta \Big( \frac {|x|}R\Big) \langle \nabla u^{1-\theta} (x), x \rangle dx = & - (n-1) \int_{B_{\partial\Rset_+^n}(0,N)} \zeta \Big( \frac {|x|}R\Big) u^{1-\theta} (x) dx \\ 
&- \int_{B_{\partial\Rset_+^n}(0,N)} \Big\langle \nabla \zeta \Big( \frac {|x|}R\Big) , x \Big\rangle u^{1-\theta} (x) dx \\
& + \int_{\partial B_{\partial\Rset_+^n} (0,N)} \zeta \Big( \frac {|x|}R\Big) u^{1-\theta} (x) \Big\langle x, \frac x{|x|}\Big\rangle d\sigma.
\end{split}
\end{equation}
For fixed $R$, by taking $N> 2R$, we deduce from the definition of $\zeta$ that the last term on the right hand side of \eqref{eq-p5-eq3} vanishes. $III =0$. Therefore, taking the limit as $N \to +\infty$ gives
\begin{equation}\label{eq-p5-eq3'}
\begin{split}
\int_{ \partial\Rset_+^n} \zeta \Big( \frac {|x|}R\Big) \langle \nabla u^{1-\theta} (x), x \rangle dx = & - (n-1) \int_{ \partial\Rset_+^n } \zeta \Big( \frac {|x|}R\Big) u^{1-\theta} (x) dx \\ 
&- \int_{ \partial\Rset_+^n } \Big\langle \nabla \zeta \Big( \frac {|x|}R\Big) , x \Big\rangle u^{1-\theta} (x) dx \\
=& I + II.
\end{split}
\end{equation}
To estimate $II$, we note by a standard computation that
\[
\Big|\Big\langle \nabla \zeta \Big( \frac {|x|}R\Big) , x \Big\rangle\Big| \leqslant \frac {2|x|}R
\]
holds. From this one can conclude that $II \to 0$ as $R \to +\infty$ since
\[\begin{split}
\int_{\partial\Rset_+^n} \Big\langle \nabla \zeta \Big( \frac {|x|}R\Big) , x \Big\rangle u^{1-\theta} (x) dx = & \int_{\partial\Rset_+^n \cap \{R \leqslant |x| \leqslant 2R\}}  \Big\langle \nabla \zeta \Big( \frac {|x|}R\Big) , x \Big\rangle u^{1-\theta} (x) dx\\
 \leqslant & 4 \int_{\partial\Rset_+^n \cap \{R \leqslant |x| \leqslant 2R\}} u^{1-\theta} dx.
\end{split}\]
Thus, by the dominated convergence theorem, we obtain
\begin{equation}\label{eq-p5-eq4}
\begin{split}
\lim_{R \to +\infty} \int_{\partial\Rset_+^n} \zeta& \Big( \frac {|x|}R\Big) \langle \nabla u^{1-\theta} (x), x \rangle dx = - (n-1) \int_{\partial\Rset_+^n} u^{1-\theta} (x) dx.
\end{split}
\end{equation}
As a consequence of \eqref{eq-p5-eq4}, we obtain
\begin{equation}\label{eq-p5-eqKeyU}
\int_{\partial\Rset_+^n} \langle \nabla u^{1-\theta} (x), x \rangle dx = - (n-1) \int_{\partial\Rset_+^n} u^{1-\theta} (x) dx.
\end{equation}
A similar argument shows
\begin{equation}\label{eq-p5-eqKeyV}
\int_{ \Rset_+^n} \langle \nabla v^{1-\kappa} (y), y \rangle dy = - n \int_{ \Rset_+^n} v^{1-\kappa} (y) dy.
\end{equation}
Notice that by the Fubini theorem, there holds
\begin{equation}\label{eq-p5-eqKeyIdentity}
\begin{split}
\frac{\kappa}{\kappa-1} \int_{\Rset_+^n} \langle \nabla v^{1-\kappa} (y), y \rangle dy =& \int_{\Rset_+^n} \langle \nabla v^{-\kappa} (y), y \rangle v(y) dy \\
=& \int_{\Rset_+^n} \langle \nabla v^{-\kappa} (y), y \rangle \int_{\partial \Rset_+^n} |x-y|^\lambda u(x)^{-\theta} dx dy\\
=& \int_{\partial \Rset_+^n} u(x)^{-\theta} \int_{\Rset_+^n} |x-y|^\lambda \langle \nabla v^{-\kappa} (y), y \rangle dy dx.
\end{split}
\end{equation}
For $\mu>0$, we set $y=\mu z$. A simple variable change tells us that
\[\begin{split}
u(\mu x) =& \int_{ \Rset_+^n} |\mu x-y|^\lambda v(y)^{-\kappa} dy = \mu^{n+\lambda} \int_{ \Rset_+^n} |x-z|^\lambda v(\mu z)^{-\kappa} dz.
\end{split}\]
Differentiating with respect to $\mu$ gives
\[\begin{split}
\langle x, (\nabla u) (\mu x) \rangle =& (n+\lambda) \mu^{n+\lambda-1} \int_{ \Rset_+^n} |x-z|^\lambda v(\mu z)^{-\kappa} dz\\
&+ \mu^{n+\lambda } \int_{ \Rset_+^n} |x-z|^\lambda \langle z, (\nabla v^{-\kappa})(\mu z) \rangle dz,
\end{split}\]
which implies, after setting $\mu =1$ and using \eqref{eqIntegralSystem}, the following
\begin{equation}\label{eq-p5-eqGradientUIdentity}
\begin{split}
\langle x, \nabla u ( x) \rangle =& (n+\lambda) u(x) + \int_{ \Rset_+^n} |x-z|^\lambda \langle z, (\nabla v^{-\kappa})( z) \rangle dz.
\end{split}
\end{equation}
Hence, by multiplying both sides of \eqref{eq-p5-eqGradientUIdentity} by $u^{-\theta}$ and making use of \eqref{eq-p5-eqKeyIdentity}, we have just shown that
\begin{equation}\label{eq-p5-eqKeyIdentity2}
\begin{split}
\int_{\partial\Rset_+^n} \langle \nabla u (x), x \rangle u^{ -\theta} (x) dx =& (n+\lambda) \int_{\partial\Rset_+^n} u^{1-\theta} (x) dx + \frac{\kappa}{\kappa-1} \int_{\Rset_+^n} \langle \nabla v^{1-\kappa} (y), y \rangle dy.
\end{split}
\end{equation}
Thanks to \eqref{eq-p5-eqKeyU} and \eqref{eq-p5-eqKeyV}, it follows from \eqref{eq-p5-eqKeyIdentity2} that
\[
- \frac{n-1}{1-\theta} \int_{\partial\Rset_+^n} u^{1-\theta} (x) dx =(n+\lambda) \int_{\partial\Rset_+^n} u^{1-\theta} (x) dx - \frac{\kappa n}{\kappa-1} \int_{\Rset_+^n} v^{1-\kappa} (y) dy.
\]
Thus, we have just proved that
\[
\Big( n+\lambda - \frac{n-1}{ 1-\theta } \Big) \int_{\partial\Rset_+^n} u^{1-\theta} (x) dx = \frac{\kappa n}{\kappa-1} \int_{\Rset_+^n} v^{1-\kappa} (y) dy.
\]
Thanks to Lemma \ref{lem-IntegralU=IntegralV}, we conclude that $(1 - 1/n)/(\theta - 1) + 1/(\kappa - 1) = \lambda /n$ as claimed.

It is worth noticing that the necessary condition \eqref{eq:NecessaryCond} is exactly the same as the condition \eqref{eq:condHLS} in the reversed HLS inequality \eqref{eq:reverseHLS} if one replaces $\kappa$ and $\theta$ by $1/(1-r)$ and $1/(1-p)$ respectively. 

\begin{lemma}\label{lem->0>0}
For $n \geqslant 1$, $\lambda>0$, $ \kappa>0$ and $\theta>0$ satisfying $\theta = \kappa - 2/\lambda$, there holds
\[
2n - \kappa \lambda + \lambda \geqslant 0, \quad 2n -2 - \theta \lambda + \lambda \geqslant 0.
\]
\end{lemma}

\begin{proof}
Suppose that $\kappa \geqslant 1 + 2n/\lambda$ and hence $\theta \geqslant 1 + (2n - 2)/\lambda$, we make use of Lemma \ref{lemNECESSARY} to conclude that $\lambda>0$ and $ \kappa>0$ fulfill \eqref{eq:NecessaryCond}, that is, $(n - 1)/(n(\theta - 1)) +1/(\kappa - 1) = \lambda /n$. Resolving this equation with the condition $\theta = \kappa - 2/\lambda$ gives $\kappa = 1 + 2n/\lambda$ and $\theta = 1 + (2n - 2)/\lambda$. From this we obtain the equalities since $2n - \kappa \lambda + \lambda = 0$ and $2n -2 - \theta \lambda + \lambda = 0$.

Otherwise, there holds $\kappa < 1 + 2n/\lambda$ and hence $\theta < 1 + (2n - 2)/\lambda$. Form this, it is immediate to see that 
$2n - \kappa \lambda + \lambda > 0$ and $2n -2 - \theta \lambda + \lambda > 0$.
\end{proof}


\section{A classification of solutions of (\ref{eqIntegralSystem}): Proof of Proposition \ref{thmCLASSIFICATION}}

Recall that $\kappa = 1 + 2n/\lambda$ and that $\theta = 1 + (2n - 2)/\lambda$, thanks to \eqref{eq:NecessaryCond} and our hypothesis $\kappa =\lambda + 2/\theta$.

\subsection{The method of moving spheres for systems}

Let $w$ be a positive function on $\overline{\Rset_+^n}$ where we denote $\overline{\Rset_+^n} = \Rset_+^n \cup \partial \Rset_+^n$. For $x \in \partial \Rset_+^n$ and $\nu>0$ we define
\begin{equation}\label{eqFunctionChange}
w_{x,\nu}( \xi ) = \big( |\xi-x| / \nu \big)^\lambda  w( \xi^{x,\nu} )
\end{equation}
for all $\xi \in \overline{\Rset_+^n}$ where $\xi^{x,\nu}$ is the Kelvin transformation of $\xi$ with respect to the ball $B_{ \Rset_+^n} (x,\nu) \subset \Rset_+^n$, given as follows
\begin{equation}\label{eqVariableChange}
{\xi^{x,\nu}} = x + \nu^2\frac{\xi - x}{{| \xi - x |}^2}.
\end{equation}
It is important to note that in the whole moving spheres arguments in this section, only spheres centered on the boundary hyperplane $\partial \Rset_+^n$ can be used. This provides a reason why we cannot capture further information for $v$ out of $\partial \Rset_+^n$ as indicated in Proposition \ref{thmCLASSIFICATION}. Clearly, upon the change of variable $y=z^{x,\nu}$ with $z\in \overline{\Rset_+^n}$, we then have
\begin{equation}\label{eqJacobian}
dy=
\begin{cases}
\left( \nu /|z-x| \right)^{2n} dz &\text{ if } z \in \Rset_+^n\\
\left( \nu /|z-x| \right)^{2n-2} dz &\text{ if } z \in \partial\Rset_+^n.
\end{cases}
\end{equation}

\begin{lemma}\label{lem0}
For any solutions $(u,v)$ of \eqref{eqIntegralSystem}, we have
\[{u_{x,\nu}}(\xi ) = \int_{\Rset_+^n} | \xi - z|^\lambda v_{x,\nu} (z)^{ - \kappa} dz\]
for any $\xi \in \partial\Rset_+^n$ and
\[{v_{x,\nu}}(\eta ) = \int_{\partial\Rset_+^n} | \eta - z|^\lambda u_{x,\nu} (z)^{ - \theta} dz \]
for any $\eta \in \Rset_+^n$.
\end{lemma}

\begin{proof}
Using our system \eqref{eqIntegralSystem}, we obtain
\[\begin{split}
{u_{x,\nu}}(\xi ) =& {\left( {\frac {|\xi - x|}{\nu}} \right)^\lambda}u({\xi ^{x,\nu}}) ={\left( {\frac {|\xi - x|}{\nu}} \right)^\lambda}\int_{\Rset_+^n} | {\xi ^{x,\nu}} - y|^\lambda v{(y)^{ - \kappa}}dy\\
 =& \int_{\Rset_+^n} | \xi - z|^\lambda {\left( {\frac{\lambda } {|z - x|} } \right)^{2n - \kappa \lambda + \lambda}}{v_{x,\nu}}{(z)^{ - \kappa}}dz.
\end{split}\]
From this we obtain the desired formula for $u$, thanks to $2n - \kappa \lambda + \lambda =0$. The formula for $v$ follows the same line as above with a little difference since we need to integrate over $\partial\Rset_+^n$.
\end{proof}

Next, we estimate $u_{x,\nu}(\xi ) - u(\xi )$ and $v_{x,\nu}(\eta ) - v (\eta )$. Since the computation is elementary and well-known in other contexts, we omit its details.

\begin{lemma}\label{lem1}
For any solutions $(u,v)$ of \eqref{eqIntegralSystem} any $\lambda >0$ and any $x \in \partial\Rset_+^n$, we have
\[{u_{x,\nu}}(\xi ) - u(\xi ) = \int_{\Sigma_{x,\nu}^n}  k(x,\nu ;\xi ,z) \big[v(z)^{-\kappa} - {v_{x,\nu}}{(z)^{ - \kappa}} \big] dz,\]
for any $\xi \in \partial\Rset_+^n$ and
\[{v_{x,\nu}}(\eta ) - v (\eta ) = \int_{\Sigma_{x,\nu}^{n-1}} k(x,\nu ;\eta ,z) \big[ u(z)^{-\theta} - {u_{x,\nu}}{(z)^{ - \theta}} \big] dz,\]
for any $\eta \in \Rset_+^n$ where
\[k(x,\nu ;\zeta ,z) = \left( {\frac {|\zeta - x|}{\nu}} \right)^\lambda |{\zeta ^{x,\nu}} - z|^\lambda - |\xi - z|^\lambda .\]
Moreover, $k(x,\nu; \zeta, z)>0$ for any $|\zeta - x| > \lambda>0$ and $|z - x| > \lambda>0$.
\end{lemma}

In the following lemma, we prove that the method of moving spheres can get started starting from a very small radius.

\begin{lemma}\label{lemStartMS}
For each $x \in \partial \Rset_+^n$, there exists some $\nu_0(x)>0$ such that for any $\nu \in (0, \nu_0(x))$
\[{u_{x,\nu}}(\xi) \geqslant u(\xi)\]
for any point $\xi \in \Sigma_{x,\nu}^{n-1}$ and
\[{v_{x,\nu}}(\eta) \geqslant v(\eta)\]
for any point $\eta \in \Sigma_{x,\nu}^n$.
\end{lemma}

\begin{proof}
Since $u$ is a positive $C^1$-function in $\partial\Rset_+^n$ and $\lambda>0$, there exists some $r_0>0$ small enough such that
\[\nabla _\xi \big( |\xi-x|^{-\lambda/2} u(\xi) \big) \cdot (\xi-x) < 0\]
for all $\xi \in \partial\Rset_+^n$ with $0<|\xi-x|<r_0$. Consequently, we can estimate
\[
\begin{split}
u_{x,\nu} (\xi) =& {\left( {\frac{|\xi-x|}{\nu}} \right)^\lambda}u (\xi^{x,\nu}) = |\xi-x|^{\lambda/2} |\xi^{x,\nu}-x|^{-\lambda/2} u (\xi^{x,\nu})>  u(\xi)
\end{split}
\]
for all $\xi \in \partial\Rset_+^n$ with $0<\lambda < |\xi-x| < r_0$. Note that in the previous estimate, we made use of the fact that if $|\xi-x| > \lambda$ then $|\xi^{x,\nu} -x| < \lambda$. Note that for small $\nu_0 \in (0, r_0)$ and for each $0<\lambda<\nu_0$, we have
\[ 
u_{x, \lambda }(\xi) \geqslant {\left( {\frac{|\xi-x|}{\nu}} \right)^\lambda}\mathop {\inf }\limits_{B(x, r_0)} u \geqslant u(\xi)
\]
for all $|\xi-x| \geqslant r_0$. Hence, we have just shown that $u_{x, \lambda }(\xi) \geqslant u(\xi)$ for all point $\xi \in \partial\Rset_+^n$ and any $\lambda$ such that $|\xi-x| \geqslant \lambda$ with $0 < \nu < \nu_0$. A similar argument also shows that $v_{x, \lambda }(\eta) \geqslant v(\eta)$ for all point $\eta \in \Rset_+^n$ and any $\lambda$ such that $|\eta-x| \geqslant \lambda$ with $0 < \lambda< \lambda_1$ for some $\lambda_1 \in (0, r_1)$. Simply setting $\nu_0 (x) = \min\{ \nu_0, \lambda_1\}$ we obtain the desired result.
\end{proof}

For each $x \in \partial \Rset_+^n$ we define
\[\overline \nu (x) = \sup \left\{ {\mu > 0:{u_{x,\nu}}(\xi) \geqslant u(\xi), {v_{x,\nu}}(\eta) \geqslant v(\eta), \forall 0 < \nu < \mu , \xi \in \Sigma_{x,\nu}^{n-1} , \eta \in \Sigma_{x,\nu}^{n} } \right\}.\]
In view of Lemma \ref{lemStartMS} above, we get $0 < \overline \nu (x) \leqslant + \infty$. In the next few lemmas, we show that whenever $\nu (x)$ is finite for some point $x$, we can write down precisely the form of $(u,v)$.

\begin{lemma}\label{lem3}
If $\overline \nu (x_0) <\infty$ for some point $x_0 \in \partial\Rset_+^n$ then
\[u_{x_0,\overline \nu (x_0)} \equiv u, \quad {v_{{x_0},\overline \nu ({x_0})}} \equiv v\]
in $\partial \Rset_+^n$ and $\Rset_+^n$, respectively. In addition, we obtain $q = 1 + 2n/p$.
\end{lemma}

\begin{proof}
By the definition of $\overline \nu (x_0)$, we know that
\begin{equation}\label{eqProof0}
{u_{x_0,\overline \nu (x_0)}}(\xi) \geqslant u(\xi), \quad {v_{x_0,\overline \nu (x_0)}}(\eta) \geqslant v(\eta)
\end{equation}
for any $\xi \in \Sigma_{x_0,\overline \nu (x_0)}^{n-1}$ and $\eta \in \Sigma_{x_0,\overline \nu (x_0)}^n$. In view of Lemma \ref{lem1}, we obtain
\begin{equation}\label{eqProof1}
\begin{split}
{u_{x_0,\overline \nu (x_0)}}(\xi) - u(\xi) = \int_{ \Sigma_{x_0,\overline \nu (x_0)}^n } & k( x_0,\overline \nu (x_0);\xi,z ) \big[v(z)^{-\kappa} - {v_{x_0,\overline \nu (x_0)}} (z)^{ -\kappa} \big] dz,
\end{split}
\end{equation}
and
\begin{equation}\label{eqProof2}
\begin{split}
{v_{x_0,\overline \nu (x_0)}}(\eta) - v(\eta) = \int_{ \Sigma_{x_0,\overline \nu (x_0)}^{n-1} } & k( x_0,\overline \nu (x_0);\eta,z ) \big[u(z)^{-\theta} - {u_{x_0,\overline \nu (x_0)}}{(z)^{ - \theta}}\big] dz.
\end{split}
\end{equation}
Keep in mind that $2n - \kappa \lambda + \lambda \geqslant 0$ and $2n-2 -\theta \lambda + \lambda \geqslant 0$ by Lemma \ref{lem->0>0}; hence there are two possible cases:

\noindent\textbf{Case 1}. Either ${u_{x_0,\overline \nu (x_0)}}(\xi) = u(\xi)$ for any $\xi \in \Sigma_{x_0,\overline \nu (x_0)}^{n-1}$ or ${v_{x_0,\overline \nu (x_0)}}(\eta) = v(\eta)$ for any $\eta \in \Sigma_{x_0,\overline \nu (x_0)}^n$. Without loss of generality, we assume that the formal case occurs. Using \eqref{eqProof1} and the positivity of the kernel $k$, we get that $2n - \kappa \lambda + \lambda=0$ and that ${v_{x_0,\overline \nu (x_0)}}(\eta) = v(\eta)$ for any $\eta \in \Sigma_{x_0,\overline \nu (x_0)}^n$. Hence by \eqref{eqProof1} we conclude that ${u_{x_0,\overline \nu (x_0)}}(\xi) = u(\xi)$ in the whole $\partial\Rset_+^n$. A similar argument also shows that $2n-2 -\theta \lambda + \lambda=0$ and that ${v_{x_0,\overline \nu (x_0)}}(\eta) = v(\eta)$ in $\Rset_+^n$ and we are done.

\noindent\textbf{Case 2}. Or ${u_{x_0,\overline \nu (x_0)}}(\xi) > u(\xi)$ for any $\xi \in \Sigma_{x_0,\overline \nu (x_0)}^{n-1}$ and ${v_{x_0,\overline \nu (x_0)}}(\eta) > v(\eta)$ for any $\eta \in \Sigma_{x_0,\overline \nu (x_0)}^n$. In this case, we derive a contradiction by showing that we can slightly move spheres a little bit over $\overline \nu (x_0)$ which then violates the definition of $\overline \nu (x_0)$.

In order to achieve that goal, first we can estimate
\begin{equation}\label{eqProof3}
\begin{split}
{u_{x_0,\overline \nu (x_0)}}(\xi) &- u(\xi) \geqslant \int_{\Sigma_{x_0,\overline \nu (x_0)}^n}  k( x_0,\overline \nu (x_0);\xi, z )\Big [v(z)^{-\kappa} - v_{x_0,\overline \nu (x_0)} (z)^{-\kappa}\Big] dz,
\end{split}
\end{equation}
thanks to the positivity of the kernel $k$.

\noindent\textbf{Estimate of $u_{x_0, \nu} - u$ outside $B_{\partial\Rset_+^n}(x_0,\overline \nu (x_0) + 1)$.} First, by the Fatou lemma and \eqref{eqProof3}, we obtain
\[\begin{split}
 \mathop {\lim \inf }\limits_{|\xi| \to +\infty }& \big( |\xi|^{-\lambda} ({u_{x_0,\overline \nu (x_0)}} - u)(\xi) \big) \\
\geqslant & \mathop {\lim \inf }\limits_{|\xi| \to +\infty } \int_{\Sigma_{x_0,\overline \nu (x_0)}^n} |\xi|^{-\lambda} k( x_0,\overline \nu (x_0);\xi,z ) \Big[v(z)^{-\kappa} - v_{x_0,\overline \nu (x_0)} (z)^{-\kappa}\Big]dz \\
 \geqslant& \int_{\Sigma_{x_0,\overline \nu (x_0)}^n} \big( {{\big(  |z| / \overline \nu (x_0)  \big) ^\lambda} - 1} \big) \big[v(z)^{-\kappa} - v_{x_0,\overline \nu (x_0)} (z)^{-\kappa} \big]dz >  0.
\end{split} \]
As a consequence, outside a large ball, we would have $({u_{x_0,\overline \nu (x_0)}} - u)(\xi) \gtrsim |\xi|^\lambda$ while in that ball and outside of $B(x_0, \overline \nu (x_0) + 1)$ we would also have $(u_{x_0,\overline \nu (x_0)} - u)(\xi) \gtrsim |\xi|^\lambda$ thanks to the smoothness of $u_{x_0,\overline \nu (x_0)} - u$ and our assumption $u_{x_0,\overline \nu (x_0)}(\xi) > u(\xi)$. Therefore, there exists some $\varepsilon_1 >0$ such that
\[({u_{x_0,\overline \nu (x_0)}} - u)(\xi) \geqslant {\varepsilon _1}|\xi|^\lambda\]
for all $|\xi-x_0| \geqslant \overline \nu (x_0) + 1$. Recall that ${u_{x_0,\overline \nu (x_0)}}(\xi) = (|x_0 - \xi|/\lambda)^\lambda u({\xi^{x_0,\overline \nu (x_0)}})$; hence there exists some $\varepsilon_2 \in (0, \varepsilon_1)$ such that
\begin{equation}\label{eqProof4}
\begin{split}
 (u_{x_0, \nu} - u)(\xi) =& ({u_{x_0,\overline \nu (x_0)}} - u)(\xi) + ({u_{x_0, \nu}} - {u_{x_0,\overline \nu (x_0)}})(\xi) \\
 \geqslant& {\varepsilon _1}|\xi|^\lambda + ({u_{x_0, \nu}} - {u_{x_0,\overline \nu (x_0)}})(\xi) \geqslant \frac{{{\varepsilon _1}}}{2}|\xi|^\lambda 
\end{split}
\end{equation}
for all $|\xi-x_0| \geqslant \overline \nu (x_0) + 1$ and all $\lambda \in (\overline \nu (x_0),\overline \nu (x_0) + {\varepsilon _2})$. Repeating the above arguments shows that \eqref{eqProof4} is also valid for $v_{x_0, \nu} - v$, that is
\begin{equation}\label{eqProof4-ForV}
\begin{split}
 (v_{x_0, \nu} - v)(\eta) \geqslant \frac{{{\varepsilon _1}}}{2}|\eta|^\lambda 
\end{split}
\end{equation}
for a possibly new constant $\varepsilon_1>0$.

\noindent\textbf{Estimate of $u_{x_0, \nu} - u$ inside $B_{\partial\Rset_+^n}(x_0,\overline \nu (x_0) + 1)$.} Now for $\varepsilon \in (0,{\varepsilon _2})$ to be determined later and for $\lambda \in (\overline \nu (x_0),\overline \nu (x_0) + \varepsilon ) \subset (\overline \nu (x_0),\overline \nu (x_0) + {\varepsilon _2})$ and for $\lambda \leqslant |\xi-x_0| \leqslant \overline \nu (x_0) + 1$, from \eqref{eqProof3}, we estimate
\[\begin{split}
 ({u_{x_0, \nu}} - u)(\xi) \geqslant& \int_{\Sigma_{x_0,\overline \nu (x_0)}^n} k( x_0, \nu;\xi,z)[v(z)^{-\kappa} - {v_{x_0, \nu}}{(z)^{ - \kappa}}]dz \\
 \geqslant &\int_{\overline \nu (x_0) + 1 \geqslant |z - x_0| \geqslant \lambda } k( x_0, \nu; \xi , z)[v(z)^{-\kappa} - {v_{x_0, \nu}}{(z)^{ - \kappa}}]dz \\
 & + \int_{\overline \nu (x_0) + 3 \geqslant |z - x_0| \geqslant \overline \nu (x_0)+ 2} k( x_0, \nu; \xi , z)[v(z)^{-\kappa} - {v_{x_0, \nu}}{(z)^{ - \kappa}}]dz \\
 \geqslant &\int_{\overline \nu (x_0) + 1 \geqslant |z - x_0| \geqslant \lambda } k( x_0, \nu; \xi , z)[v_{x_0,\overline \nu (x_0)} (z)^{-\kappa} - {v_{x_0, \nu}}{(z)^{ - \kappa}}]dz \\
 &+ \int_{\overline \nu (x_0) + 3 \geqslant |z - x_0| \geqslant \overline \nu (x_0)+ 2} k( x_0, \nu; \xi , z)[v(z)^{-\kappa} - {v_{x_0, \nu}}{(z)^{ - \kappa}}]dz \\
=& I + II. 
\end{split} \]
As we shall see later, there holds $I+II \geqslant 0$ provided $\varepsilon >0$ is small enough. To see this, we estimate $I$ and $II$ term by term.

\noindent\textbf{Estimate of $II$}. Thanks to \eqref{eqProof4-ForV}, there exists $\delta_1>0$ such that $\big( v^{-\kappa} - v_{x_0,\nu}^{-\kappa} \big)(z) \geqslant \delta_1$ for any $\overline \nu (x_0) +2 \leqslant |z-x_0| \leqslant \overline \nu (x_0) +3$. By the definition of $k$ given in Lemma \ref{lem1} we note that 
\[
k(x_0,\nu; \xi ,z)=k(0, \nu; \xi -x_0,z -x_0)
\]
and that
\[
\nabla _\xi k(0, \nu ;\xi ,z) \cdot y \big|_{|\xi| = \nu } = p |\xi -z| ^{\lambda-2} \big(|z|^2 - |\xi|^2\big) > 0
\]
for all $\overline \nu (x_0) +2 \leqslant |z| \leqslant \overline \nu (x_0) +3$. Hence, there exists some constant $\delta_2>0$ independent of $\varepsilon$ such that
\[
k(0, \nu ;\xi,z) \geqslant {\delta _2}(|\xi| - \nu )
\]
for all $\overline \nu (x_0) \leqslant \nu \leqslant |\xi| \leqslant \overline \nu (x_0) + 1$ and all $\overline \nu (x_0) + 2 \leqslant |z| \leqslant \overline \nu (x_0) + 3$. Simply replacing $y$ by $y-x_0$ and $z$ by $z-z_0$ and making use of the rule $k(x_0,\nu; \xi,z)=k(0, \nu; \xi-x_0,z -x_0)$, we obtain with the same constant $\delta_2>0$ as above the following estimate
\[k( x_0, \nu; \xi , z) \geqslant \delta_2 (|\xi-x_0| - \nu )\]
for all $\overline \nu (x_0) \leqslant \nu \leqslant |y-x_0| \leqslant \overline \nu (x_0) + 1$ and all $\overline \nu (x_0) + 2 \leqslant |z-x_0| \leqslant \overline \nu (x_0) + 3$. Thus, we have just proved that
\begin{equation}\label{eqEstimateII}
II \geqslant \delta_1 \delta_2 (|\xi-x_0|-\nu) \int_{\overline \nu (x_0) + 3 \geqslant |z - x_0| \geqslant \overline \nu (x_0) + 2} dz.
\end{equation}

\noindent\textbf{Estimate of $I$}. To estimate $I$, we first observe that $| v_{x_0, \nu } ^{ - \kappa} - {v^{ - \kappa}}|(z) \lesssim \nu - \overline \nu (x_0) \lesssim \varepsilon$ for all $z$ satisfying $\overline \nu (x_0) \leqslant \nu \leqslant |z-x_0| \leqslant \overline \nu (x_0) + 1$ and all $\overline \nu (x_0) \leqslant \nu \leqslant \overline \nu (x_0) + \varepsilon$ and that
\[\begin{split}
 \int_{\lambda \leqslant |z-x_0| \leqslant \overline \nu (x_0)+ 1} & {k( x_0, \nu; \xi , z)dz} \\
= & \int_{\lambda \leqslant |z| \leqslant \overline \nu(x_0) + 1} {k(0, \nu ; \xi-x_0,z)dz}\\
\leqslant & \int_{\lambda \leqslant |z| \leqslant \overline \nu (x_0)+ 1} \Big| \big(  |\xi-x_0| /\lambda \big)^\lambda-1 \Big| { |(\xi-x_0)^{0,\lambda } - z|^\lambda dz} \\
& + \int_{\lambda \leqslant |z| \leqslant \overline \nu(x_0) + 1} { \big(|(\xi-x_0)^{0,\lambda } - z|^\lambda - |(\xi-x_0) - z|^\lambda \big)dz} \\
 \leqslant& C(|\xi-x_0| - \lambda ) + C|{(\xi-x_0)^{0,\lambda }} - (\xi-x_0)| \\
 \leqslant & C(|\xi-x_0| - \lambda ).
\end{split} \]
where $C>0$ is constant independent of $\varepsilon$.
Thus, we obtain
\begin{equation}\label{eqEstimateI}
I \geqslant -C\varepsilon \int_{\overline \nu (x_0) + 1 \geqslant |z - x_0| \geqslant \lambda } k( x_0, \nu; \xi , z) dz.
\end{equation}
By combining \eqref{eqEstimateI} and \eqref{eqEstimateII}, for some small $\varepsilon>0$ we have
\[\begin{split}
 ({u_{x_0, \nu }} - u)(\xi) \geqslant& \Big( \delta_1 \delta_2 \int_{\overline \nu (x_0) + 3 \geqslant |z - x_0| \geqslant \overline \nu (x_0) + 2} dz -C\varepsilon \Big) (|\xi-x_0|-\nu) \geqslant 0
\end{split} \]
for $\overline \nu (x_0) \leqslant \nu \leqslant \overline \nu (x_0) + \varepsilon$ and $\nu \leqslant |y-x_0| \leqslant \overline \nu (x_0) + 1$. 

\noindent\textbf{Estimates of $u_{x_0, \nu} - u$ in $B_{\partial\Rset_+^n}(x_0,\overline \nu (x_0) + 1)$ and $v_{x_0, \nu} - v$ in $B_{\Rset_+^n}(x_0,\overline \nu (x_0) + 1)$.} Combining the preceding estimate for $u_{x_0, \nu} - u$ inside the ball $B(x_0, \overline \nu (x_0) + 1)$ and \eqref{eqProof4} above gives
\[\begin{split}
 ({u_{x_0, \nu}} - u)(\xi) \geqslant 0
\end{split} \]
for $\overline \nu (x_0) \leqslant \nu \leqslant \overline \nu (x_0) + \varepsilon$ and $\nu \leqslant |y-x_0|$. Again by repeating the whole procedure above for the difference $v_{x_0,\nu} -v$, we can conclude that 
$$(v_{x_0,\nu} -v)(y) \geqslant 0$$ 
for $\overline \nu (x_0) \leqslant \nu \leqslant \overline \nu (x_0) + \varepsilon$ and $\nu \leqslant |y-x_0|$ where $\varepsilon$ could be smaller if necessary; thus giving us a contradiction to the definition of $\overline\nu (x_0)$.
\end{proof}

In the last lemma of the current section, we prove that whenever $\overline \nu (x_0) <\infty$ for some point $x_0 \in \partial \Rset_+^n$, there must hold $\overline \nu (x) <\infty$ for any point $x \in \partial \Rset_+^n$.

\begin{lemma}\label{lemLamdaVanishAtEveryPoint}
If $\overline \nu (x_0) <\infty$ for some point $x_0 \in \partial \Rset_+^n$ then $\overline \nu (x) <\infty$ for any point $x \in \partial \Rset_+^n$; hence
\[{u_{x,\overline \nu (x)}} \equiv u, \quad {v_{x,\overline \nu (x)}} \equiv v\]
for all $x \in \partial \Rset_+^n$.
\end{lemma}

\begin{proof}
Suppose that there exists some $x_0 \in \partial \Rset_+^n$ such that $\overline\nu (x_0)<\infty$, by Lemma \ref{lem3} and for $\xi \in \partial\Rset_+^n$ with $|\xi|$ sufficiently large, we have
\begin{align*}
 {| \xi |^{-\lambda} }u(y) &= | \xi |^{-\lambda} u_{x_0,\overline \nu (x_0)} (\xi)\\
& = {| \xi |^{-\lambda} }{\left( {\frac{{\overline \nu ({x_0})}}{{\left| {\xi - {x_0}} \right|}}} \right)^{-\lambda} }  u\Big( {{x_0} + \lambda {{({x_0})}^2}\frac{\xi - x_0}{{{{\left| {\xi - x_0} \right|}^2}}}} \Big) \\
 &= \overline \nu {({x_0})^{-\lambda}}{\left( {\frac{{\left| \xi- x_0 \right|}}{| \xi |}} \right)^\lambda } u\Big( {{x_0} + \lambda {{({x_0})}^2}\frac{{\xi - x_0}}{{{{\left| {\xi - x_0} \right|}^2}}}} \Big) 
\end{align*}
which implies
\begin{equation}\label{eq14}
\mathop {\lim }\limits_{| \xi | \to +\infty } {| \xi |^{-\lambda} }u(\xi) = \overline \nu {({x_0})^{-\lambda}}u({x_0}).
\end{equation}
Repeating the above argument then gives
\begin{equation}\label{eq15}
\mathop {\lim }\limits_{| \eta | \to +\infty } {| \eta |^{-\lambda} } v(\eta) = \overline \nu {({x_0})^{-\lambda}} v(x_0).
\end{equation}
Let $x \in \partial \Rset_+^n$ be arbitrary, by the definition of $\overline\nu (x)$ we get
\[u_{x,\nu} (\xi) \geqslant u(\xi), \quad v_{x,\nu} (\eta) \geqslant v(\eta) \]
for all $0 < \nu < \overline \nu (x)$ and all $\xi \in \Sigma_{x, \nu }^{n-1}$ and $\eta \in \Sigma_{x, \nu }^n$. Then by a direct computation and thanks to \eqref{eq14}, one can easily see that
\begin{equation}\label{eq16}
\begin{split}
\mathop {\liminf }\limits_{| \xi | \to +\infty } {| \xi |^{-\lambda}}u(\xi) &\leqslant \mathop {\lim \inf }\limits_{| \xi | \to +\infty } {| \xi |^{-\lambda}}{u_{x,\nu}}(\xi) \\
&= \mathop {\liminf }\limits_{| \xi | \to +\infty } {| \xi |^{-\lambda} }{\left( {\frac{\nu}{{\left| {\xi - x} \right|}}} \right)^{-\lambda} }u\Big( {x + \nu^2 \frac{\xi-x}{{{{\left| {\xi - x } \right|}^2}}}} \Big) \\
&= \nu ^{-\lambda} u(x)
\end{split}
\end{equation}
for all $0 < \lambda < \overline \nu (x)$. Combining \eqref{eq14} and \eqref{eq16} gives $ \overline \nu (x_0)^{-\lambda} u( x_0 ) \leqslant \nu ^{-\lambda} u(x)$ for all $0 < \lambda < \overline \nu (x)$. From this, we conclude $\overline \nu (x) < +\infty$ for all $x \in {\Rset^n}$ as claimed.
\end{proof}

\subsection{Proof of Proposition \ref{thmCLASSIFICATION}}

To conclude Proposition \ref{thmCLASSIFICATION}, we make use of the following three lemmas from \cite[Appendix B]{lz2003}. The first lemma concerns functions $f$ satisfying the inequality $f_{x,\nu}(\xi) \leqslant f(\xi)$, which can be the case in view of Lemma \ref{lemStartMS}.

\begin{lemma}\label{lemKey1}
For $\nu \in \Rset$ and $f$ a function defined on $\partial\Rset_+^n$ and valued in $[-\infty, +\infty]$ satisfying
\[
\Big( \frac{\nu }{ | \xi - x | } \Big)^\lambda f\Big( x + \nu^2\frac{\xi - x}{ | \xi - x |^2} \Big) \leqslant f(\xi)
\]
for all $x, \xi \in \partial\Rset_+^n$ satisfying $| x - \xi | > \nu > 0$. Then $f$ is constant or is identical to infinity.
\end{lemma}

In the second lemma, if the inequality in Lemma \ref{lemKey1} becomes equality, then we can characterize the function $f$ completely; see also Lemma \ref{lem3}. 

\begin{lemma}\label{lemKey2}
For $\nu \in \Rset$ and $f$ a continuous function in $\partial\Rset_+^n$. Suppose that for every $x\in \partial\Rset_+^n$, there exists $\lambda(x)>0$ such that
\[ \Big( {\frac{{\nu (x)}}{ | \xi - x | }} \Big)^\lambda  f\Big( {x + \nu (x)^2\frac{\xi - x}{{| \xi - x |}^2}} \Big) = f(\xi),\]
for all $\xi \in \partial\Rset_+^n \backslash \{ x\} $. Then
\[
f(x) = \pm a{\big( d + |  x - \overline x  |^2  \big)^{-\lambda /2}}
\]
for some $a \geqslant 0$, $d>0$ and $\overline x \in \partial \Rset_+^n$.
\end{lemma}

It is worth noting that Lemma \ref{lemKey2} also holds for a larger class consisting measures, known as the characterization of inversion invariant measures; see \cite[Theorem 1.4]{fl2010}. The last lemma is in the same fashion of Lemma \ref{lemKey2} above for functions defined on $\overline{\Rset_+^n}$.

\begin{lemma}\label{lemKey3}
For $\nu \in \Rset$ and $f$ a function defined on $\overline{\Rset_+^n}$ and valued in $[-\infty, +\infty]$ satisfying
\[
\Big( {\frac{\nu }{ | y - x | }} \Big)^\lambda f\Big( {x + \nu^2\frac{{y - x}}{{| y - x |}^2}} \Big) \leqslant f(y)
\]
for all $x \in \partial \Rset_+^n$ and $y \in \Rset_+^n$ satisfying $| x - y | > \nu > 0$. Then $f$ restricted to $\partial \Rset_+^n$ is constant or is identical to infinity. In other words, $f$ depends only on the last coordinate.
\end{lemma}

With all ingredients above, we are now in a position to prove Proposition \ref{thmCLASSIFICATION}. First, there are two possible cases:

\noindent\textbf{Case 1.} If $\overline\nu (x)=\infty$, then for any $x \in \partial \Rset_+^n$ we know that $u_{x,\nu}(\xi) \geqslant u(\xi)$ for all $\lambda > 0$ and for any $x \in \partial \Rset_+^n$ and $\xi \in \partial \Rset_+^n$ satisfying $|\xi-x| \geqslant \lambda$. By Lemma \ref{lemKey1}, $u$ must be constant, say $u_0$. Similarly, Lemma \ref{lemKey3} tells us that $v$ depends only on the last coordinate in the sense that
\[
v(y) = v(0, y_n) = \int_{\partial \Rset_+^n} \big| |x|^2 + y_n^2 \big|^\lambda (u_0)^{-\theta} dx .
\]
Upon a change of variables, we deduce that
\[
v(y) = v(0, y_n) = C y_n^{2\lambda + n-1} \int_0^{+\infty} \big| \rho^2 + 1 \big|^\lambda \rho^{n-2} d\rho .
\]
From this we conclude that $v(0, y_n) = +\infty$ provided $y_n\ne 0$ since $\lambda>0$ and $n-2 \geqslant 0$. Thus, $(u,v)$ does not solve \eqref{eqIntegralSystem}. 

\noindent\textbf{Case 2.} Otherwise, there exists some $x_0 \in \partial \Rset_+^n$ such that $\overline\nu (x_0)<\infty$. Then by Lemma \ref{lemLamdaVanishAtEveryPoint}, we deduce that $\overline\nu (x)<\infty$ for any point $x \in \partial \Rset_+^n$. We are now in a position to apply Lemma \ref{lemKey2} to conclude that $u$ is of the form
\begin{equation}\label{eqForm4U}
u(x) = a ( b ^2 + | {x - \overline x } | ^2)^{\lambda/2}
\end{equation}
for some $a , d >0$ and some point $\overline x \in \partial \Rset_+^n$. Using the form of $u$ and the equation satisfied by $v$ in \eqref{eqIntegralSystem}, we get
\[
v(y) = a \int_{\partial \Rset_+^n} \frac{|x-y|^\lambda}{ ( b ^2 + | x - \overline x | ^2)^{n-1+ \lambda/2} } dx
\]
for any $y \in \Rset_+^n$. If we restrict $y$ to $\partial \Rset_+^n$, we clearly get
\[
u(y|_{\partial \Rset_+^n}) = a\int_{\partial \Rset_+^n} \frac{|x-y|_{\partial \Rset_+^n}|^\lambda}{ ( b^2 + | {x - \overline x } | ^2)^{n-1+ \lambda/2} } dx.
\]
From this, we obtain $v(y|_{\partial \Rset_+^n}) = u(y|_{\partial \Rset_+^n})$. Hence, $v$ is of the following form
\begin{equation}\label{eqForm4V}
v(x,0)= a ( b^2 + | x - \overline x | ^2)^{\lambda/2}.
\end{equation}
Our proof of Proposition \ref{thmCLASSIFICATION} is now complete.

Finally, we conclude this section by giving the proof of Corollary \ref{explicit}. Recall that under the current situation we have $\lambda =2$. By Proposition \ref{thmCLASSIFICATION} and up to a constant multiplication, translation, and dilation, we know that the extremal function $f$ has the following form
\[
f(x) = (1 +|x|^2)^{-n},\quad x\in \partial \Rset_+^n,
\]
and thanks to our system
\[
g(y) = \alpha \Big(\int_{\partial\Rset_+^n} |x-y|^2 f(x) dx\Big)^{-n-1},\quad y \in \Rset_+^n,
\]
for some constant $\alpha >0$. An easy computation shows that
\begin{equation}\label{eq:comp}
\int_{\partial\Rset_+^n} (1+|x|^2)^{-n} dx = \int_{\partial\Rset_+^n} (1+|x|^2)^{-n} |x|^2 dx = \frac{\pi^{(n-1)/2} \Gamma((n+1)/2)}{\Gamma(n)}.
\end{equation}
Therefore we conclude that 
$$g(y) = \beta (1+|y|^2)^{-n-1}$$ 
for some $\beta > 0$. Denote $h(y) = (1+|y|^2)^{-n-1}$, it is evident that
\begin{equation}\label{eq:a0}
\Cscr_{n,1-1/n, n/(n+1)}^+ = \frac{\int_{\Rset_+^n}\int_{\partial\Rset_+^n} f(x)  |x-y|^2 h(y) dxdy}{\|f\|_{L^{(n-1)/n}(\partial\Rset_+^n)} \|h\|_{L^{n/(n+1)}(\Rset_+^n)}}.
\end{equation}
We now estimate \eqref{eq:a0} term by term. First, in view of \eqref{eq:comp}, we have
\begin{equation}\label{eq:a1}
\int_{\Rset_+^n}\int_{\partial\Rset_+^n}  f(x)  |x-y|^2 h(y) dxdy = \frac{\pi^{(n-1)/2} \Gamma((n+1)/2)}{\Gamma(n)}\int_{\Rset_+^n} (1+|y|^2)^{-n} dy.
\end{equation}
Regarding to $\|f\|_{L^{(n-1)/n}(\partial\Rset_+^n)}$ and $\|h\|_{L^{n/(n+1)}(\Rset_+^n)}$, it is easy to check that
\begin{equation}\label{eq:a2}
\int_{\partial\Rset_+^n} (1+|x|^2)^{-n+1} dx = \pi^{(n-1)/2} \frac{\Gamma((n-1)/2)}{\Gamma(n-1)},
\end{equation}
and that
\begin{equation}\label{eq:a3}
\int_{\Rset_+^n} (1+|y|^2)^{-n} dy = \frac{\pi^{n/2}}2 \frac{\Gamma(n/2)}{\Gamma(n)}.
\end{equation}
Plugging \eqref{eq:a1}, \eqref{eq:a2} and \eqref{eq:a3} into \eqref{eq:a0}, we obtain \eqref{eq:explicit} as claimed.


\section{A log-HLS inequality on half space: Proof of Theorem \ref{log-HLSonhalfspace}}

Throughout this section, we choose $p = 2(n-1)/(2(n-1) + \lambda)$ and $r = 2n/(2n +\lambda)$. The benefit of this particular choice for $p$ and $r$ is that there exists an optimizer pair $(f,g)$ with $f(x)=u(x)^{1/(p-1)}=(1+|x|^2)^{1-n - \lambda/2}$ on $\partial \Rset_+^n$ and $g(y',0)=v(y',0)^{1/(r-1)}=(1+|y'|^2)^{-n-\lambda/2}$, up to dilations and translations. Consider the stereographic projection $\mathcal{S}: \Rset^n \to \mathbb S^n$ given by
\[
\mathcal{S}(x) = \Big(\frac{2x_1}{1+|x|^2},\frac{2x_2}{1+|x|^2},\ldots, \frac{1-|x|^2}{1+|x|^2}, \frac{2x_n}{1+|x|^2}\Big).
\]
It is easy to see that $\mathcal{S}$ transforms $\Rset_+^n$ into $\mathbb S^n_+ = \{(\xi_1,\cdots,\xi_{n+1})\in \mathbb S^n\, :\, \xi_{n+1} > 0\}$ and transforms $\partial\Rset_+^n$ into
 $\mathbb S^n_0 = \{(\xi_1,\cdots,\xi_{n+1})\in \mathbb S^n\, :\, \xi_{n+1} = 0\}$ which can be identified with $\mathbb S^{n-1}$. It is well-known that the Jacobian of $\mathcal{S}$ at any $x\in \Rset^n$ is 
 \[
 J_{\mathcal{S}}(x) = \big(2/(1+|x|^2) \big)^{n} 
 \]
 and the Jacobian of $\mathcal{S}|_{\partial\Rset_+^n}$ at any $y\in \partial\Rset_+^n$ is nothing but
 \[
 J_{\mathcal{S}|_{\partial\Rset_+^n}}(y) = \big(2/(1+|y|^2)\big)^{n-1}.
 \]
Let $f,g$ be nonnegative functions on $\partial\Rset_+^n$ and $\Rset_+^n$, respectively. We let $F:\mathbb S^n_0 \to \Rset$ and $G: \mathbb S^n_+ \to \Rset$ be functions given by the following
 \[
 f(y) = F(\mathcal{S}|_{\partial\Rset_+^n}(y)) J_{\mathcal{S}|_{\partial\Rset_+^n}}(y)^{1/p},\quad y\in \partial\Rset_+^n,
 \]
 and
 \[
 g(x) = G(\mathcal{S}(x)) J_{\mathcal{S}}(x)^{1/r},\quad x\in \Rset_+^n.
 \]
 We then can readily check that
\begin{equation}\label{eq:changevariable}
 \int_{\partial\Rset_+^n} f(y)^p dy = \int_{\mathbb S^n_0} F(\eta)^p d\eta,\quad \int_{\Rset_+^n} g(x)^r dx = \int_{\mathbb S^n_+} G(\xi)^r d\xi,
 \end{equation}
 where $d\xi$ and $d\eta$ denote the Lebesgue measure on $\mathbb S^n$ and on $\mathbb S_0^n$, respectively. (Note that the Lebesgue measure on $\mathbb S_0^n$ is also the Lebesgue measure on the sphere $\mathbb S^{n-1}$.) Recall that for $x,y\in \Rset^n$, there holds
\begin{equation}\label{eq:crucialequality}
|\mathcal{S}(x) - \mathcal{S}(y)|^2 = \frac2{1+|y|^2} |x-y|^2 \frac{2}{1+|x|^2}.
\end{equation}
From \eqref{eq:crucialequality}, we easily imply that
\begin{equation}\label{eq:transferonsphere}
\int_{\Rset^n_+} \int_{\partial\Rset_+^n} g(x) |x-y|^\lambda f(y) dx dy = \int_{\mathbb S^n_+}\int_{\mathbb S_0^n} G(\xi) |\xi-\eta|^\lambda F(\eta) d\xi d\eta.
\end{equation}
Combining \eqref{eq:reverseHLS}, \eqref{eq:changevariable}, and \eqref{eq:transferonsphere} gives us a spherical forms of the reversed HLS on the half space as follows
\begin{equation}\label{eq:sphericalform}
\int_{\mathbb S^n_+}\int_{\mathbb S_0^n} G(\xi) |\xi-\eta|^\lambda F(\eta) d\xi d\eta \geqslant C(n,\lambda) \|F\|_{L^p(\mathbb S^n_0)} \|G\|_{L^r(\mathbb S^n_+)},
\end{equation}
for any nonnegative functions $F\in L^p(\mathbb S^n_0)$ and $G\in L^r(\mathbb S^n_+)$ with the sharp constant
\[
C(n,\lambda) =  \|E_\lambda f_0\|_{L^q(\Rset_+^n)} \|f_0\|_{L^p(\partial\Rset_+^n)}^{-1}
\]
and the optimizer function
\[
f_0(y) = \big(1+|y|^2\big)^{1-n-\lambda/2}.
\] 
Similar to the classical HLS inequality on $\Rset^n$ which can be lifted to the sphere $\mathbb S^n$ for which the competing symmetries argument can be used. In view of the spherical form \eqref{eq:sphericalform} of \eqref{eq:reverseHLS}, it seems that a competing symmetries argument could work in this scenario. Recall that $q = -2n/\lambda$, from this it is easy to check that
\begin{equation}\label{eq:constantonsphere}
C(n,\lambda) = |\mathbb S^{n-1}|^{-\lambda/2(n-1)} \Big(\int_{\mathbb S^n_+}\Big(\int_{\mathbb S^n_0} |\xi -\eta|^\lambda d\sigma(\eta)\Big)^{-2n/\lambda} d\xi\Big)^{-\lambda/2n},
\end{equation}
where $d\sigma$ is the normalization of the Lebesgue measure on $\mathbb S^n_0$ in the sense that $\sigma(\mathbb S^n_0) = 1$.

The next Proposition gives us the behavior of the constant $C(n,\lambda)$ when $\lambda$ tends to zero.

\begin{proposition}\label{constantnearzero}
Let $n\geqslant 2$, there holds
\begin{equation}\label{eq:nearzero}
C(n,\lambda) = 1-\frac\lambda{2(n-1)} \ln |\mathbb S^{n-1}| -\frac\lambda{2n} \ln\Big(\int_{\mathbb S^n_+} e^{-2n \int_{\mathbb S^n_0} \ln(|\xi-\eta|) d\sigma(\eta)} d\xi\Big) + o(\lambda)
\end{equation}
with the error $o(\lambda)/\lambda \to 0$ as $\lambda\to 0$.
\end{proposition}

\begin{proof}
We first observe that there exists a constant $C > 0$ which is independent of $\xi\in \mathbb S^n_+$ such that 
\[
\int_{\mathbb S^n_0} \big(\ln |\xi -\eta| \big)^2 d\sigma(\eta) \leqslant C
\]
for all $\xi \in \mathbb S^n_+$. Hence
\[
\int_{\mathbb S^n_0} |\xi-\eta|^\lambda d\sigma(\eta) = 1 + \lambda \int_{\mathbb S^n_0} \ln |\xi-\eta| d\sigma(\eta) + o(\lambda)
\]
uniformly in $\xi \in \mathbb S^n_+$ when $\lambda\to 0$. For the sake of simplicity, let us denote
\[
H(\xi) = \int_{\mathbb S^n_0} \ln |\xi-\eta|  d\sigma(\eta)
\]
for each $\xi \in \mathbb S^n_+$. Since $H$ is bounded, we obtain
\[
\int_{\mathbb S^n_0} |\xi-\eta|^\lambda d\sigma(\eta) = (1 + \lambda H(\xi))(1 + o(\lambda))
\]
uniformly in $\xi \in \mathbb S^n_+$ when $\lambda\to 0$. Therefore, we have
\[
\Big(\int_{\mathbb S^n_+}\Big(\int_{\mathbb S^n_0} |\xi -\eta|^\lambda d\sigma(\eta)\Big)^{-2n/\lambda} d\xi\Big)^{-\lambda/2n} = (1+ o(\lambda)) \Big(\int_{\mathbb S^n_+}\left(1+\lambda H(\xi)\right)^{-2n/\lambda} d\xi\Big)^{-\lambda/2n}.
\]
Thanks to the boundedness of $H$, we then have
\[
\big(1+ \lambda H(\xi) \big)^{-2n/\lambda} = e^{-2n H(\xi)} (1 + o(1))
\]
uniformly in in $\xi \in S^n_+$ when $\lambda\to 0$. Hence 
\[
\Big(\int_{\mathbb S^n_+}\big(1+\lambda H(\xi)\big)^{-2n/\lambda} d\xi\Big)^{-\lambda/2n} = (1+o(\lambda)) \Big(\int_{\mathbb S^n_+} e^{-2n H(\xi)} d\xi\Big)^{-\lambda/2n}.
\]
Finally, we have
\begin{equation}\label{eq:nearzero1}
\Big(\int_{\mathbb S^n_+}\Big(\int_{\mathbb S^n_0} |\xi -\eta|^\lambda d\sigma(\eta)\Big)^{-2n/\lambda} d\xi\Big)^{-\lambda/2n} = (1 +o(\lambda)) \Big(\int_{\mathbb S^n_+} e^{-2n H(\xi)} d\xi\Big)^{-\lambda/2n}.
\end{equation}
The expansion \eqref{eq:nearzero} now follows from \eqref{eq:constantonsphere} and \eqref{eq:nearzero1}.
\end{proof}

We are now in a position to prove Theorem \ref{log-HLSonhalfspace}.

\begin{proof}[Proof of Theorem \ref{log-HLSonhalfspace}]
It is clear that $p,r\to 1$ when $\lambda \to 0$, and by Proposition \ref{constantnearzero} we have $C(n,\lambda)\to 1$ when $\lambda\to 0$. Our assumptions on $f$ and $g$ ensure that
\[
\lim_{\lambda\to 0} \frac 1\lambda \Big( \int_{\Rset_+^n}\int_{\partial\Rset_+^n} g(x) |x-y|^\lambda f(y) dx dy - 1 \Big) = \int_{\Rset_+^n}\int_{\partial\Rset_+^n} g(x) \ln |x-y| f(y) dx dy.
\]
From this, we obtain
\begin{align*}
\lim_{\lambda\to 0} &\frac{C(n,\lambda)\|f\|_{L^p(\partial\Rset_+^n)}\|g\|_{L^r(\Rset_+^n)}-1}\lambda\\
& =-\int_{\partial\Rset_+^n} f(y) \ln f(y) dy - \int_{\Rset_+^n} g(x) \ln g(x) dx \\
&\quad -\frac1{2(n-1)}\ln |\mathbb S^{n-1}|  -\frac1{2n} \ln\Big(\int_{\mathbb S^n_+} e^{-2n \int_{\mathbb S^n_0} \ln |\xi-\eta| d\sigma(\eta)} d\xi\Big).
\end{align*}
It is easy to see that
\[
C_n = -\frac1{2(n-1)}\ln |\mathbb S^{n-1}|  -\frac1{2n} \ln\Big(\int_{\mathbb S^n_+} e^{-2n \int_{\mathbb S^n_0} \ln |\xi-\eta| d\sigma(\eta)} d\xi\Big).
\] 
The proof of Theorem \ref{log-HLSonhalfspace} is then finished.
\end{proof}

\section*{Acknowledgments}

The authors want to thank Professor Meijun Zhu for stimulating discussion on the references \cite{dz2013, dz2014, hanzhu} as well as for sending us his preprint \cite{zhu2015}. Both authors are very grateful to the three anonymous referees for careful reading of our manuscript and for critical comments and suggestions which substantially improved the exposition of the article. Especially, we are indebted to one of the three referees for bringing the usage of the Gegenbauer polynomials as well as the spherical reflection positivity to out attention. V.H. Nguyen is supported by a grant from the European Research Council (grant number $305629$).


\begin{thebibliography}{9999999}

\bibitem[Alek58]{a1958}
\textsc{A.D. Aleksandrov},
Uniqueness theorems for surfaces in the large. V. 
\textit{Vestnik Leningrad. Univ.} \textbf{13} (1958), pp. 5--8.

\bibitem[Bec93]{beckner1993}
\textsc{W. Beckner},
Sharp Sobolev inequalities on the sphere and the Moser-Trudinger inequality,
\textit{Ann. of Math.} {\bf 138} (1993) 213-242.

\bibitem[Bur09]{Bur}
\textsc{A. Burchard},
\textit{A short course on rearrangement inequalities}, June 2009.
Available at: \texttt{http://www.math.utoronto.ca/almut/rearrange.pdf.}

\bibitem[CGS89]{cgs1989}
\textsc{L. Caffarelli, B. Gidas, J. Spruck},
Asymptotic symmetry and local behavior of semilinear elliptic equations with critical Sobolev growth,
\textit{Comm. Pure Appl. Math.} \textbf{42} (1989), pp. 271--297.

\bibitem[CL92]{CarlenLoss92}
\textsc{E. Carlen, M. Loss},
Competing symmetries, the logarithmic HLS inequality and Onofri's inequality on $S^n$,
\textit{Geom. Funct. Anal.} {\bf 2} (1992), pp. 90--104.

\bibitem[CL91]{cl1991}
\textsc{W. Chen, C. Li},
Classification of solutions of some nonlinear elliptic equations.,
\textit{Duke Math. J.} \textbf{63} (1991), pp. 615--622.

\bibitem[CLO05]{clo2005}
\textsc{W. Chen, C. Li, B. Ou},
Classification of solutions for a system of integral equations,
\textit{Comm. Partial Differential Equations} \textbf{30} (2005), pp. 59--65.

\bibitem[CLO06]{clo2006}
\textsc{W. Chen, C. Li, B. Ou},
Classification of solutions for an integral equation,
\textit{Comm. Pure Appl. Math.} \textbf{59} (2006), pp. 330--343.

\bibitem[DZ15h]{dz2013}
\textsc{J. Dou, M. Zhu},
Sharp Hardy--Littlewood--Sobolev inequality on the upper half space,
\textit{Int. Math. Res. Not.} (2015) Vol. 2015, Issue 3, pp. 651--687.

\bibitem[DZ15]{dz2014}
\textsc{J. Dou, M. Zhu},
Reversed Hardy--Littlewood--Sobolev inequality,
\textit{Int. Math. Res. Not.} (2015) Vol. 2015, Issue 19, pp. 9696--9726.

\bibitem[FL10]{fl2010}
\textsc{R.L. Frank, E.H. Lieb},
Inversion positivity and the sharp Hardy--Littlewood--Sobolev inequality,
\textit{Calc. Var. Partial Differential Equations} \textbf{39} (2010), pp. 85--99. 

\bibitem[FL11]{fl2011}
\textsc{R.L. Frank, E.H. Lieb},
Spherical reflection positivity and the Hardy-Littlewood-Sobolev inequality,
\textit{Concentration, functional inequalities and isoperimetry}, 89--102, 
Contemp. Math., \textbf{545}, 
Amer. Math. Soc., Providence, RI, 2011.

\bibitem[FL12a]{fl2012a}
\textsc{R.L. Frank, E.H. Lieb},
Sharp constant in several inequalities on Heisenberg group,
\textit{Ann. of Math.} {\bf 176} (2012) 349-381.

\bibitem[GNN79]{gnn1979}
\textsc{B. Gidas, W.M. Ni, L. Nirenberg}, 
Symmetry and related properties via the maximum principle, 
\textit{Comm. Math. Phys.} \textbf{68} (1979), pp. 209--243.

\bibitem[Gro75]{gross1975}
\textsc{L. Gross},
Logarithmic Sobolev inequality,
\textit{Amer. J. Math.} {\bf 97} (1976) 1061-1083.

\bibitem[HZ15]{hanzhu}
\textsc{Y. Han, M. Zhu},
Hardy--Littlewood--Sobolev inequalities on compact Riemannian manifolds and applications,
\textit{J. Differential Equations} \textbf{260} (2016), pp. 1-25.

\bibitem[HL28]{hl1928}
\textsc{G.H. Hardy, J.E. Littlewood},
Some properties of fractional integrals. I,
\textit{Math. Z.} {\bf 27} (1928) 565-606.

\bibitem[HL30]{hl1930}
\textsc{G.H. Hardy, J.E. Littlewood},
Notes on the theory of series (XII): On certain inequalities connected with the calculus of variations,
\textit{J. London Math. Soc.} {\bf 5} (1930) 34-39.

\bibitem[Lei15]{lei2015}
\textsc{Y. Lei},
On the integral systems with negative exponents,
\textit{Discrete Contin. Dyn. Syst.} \textbf{35} (2015), no. 3, pp. 1039--1057. 

\bibitem[Li04]{l2004}
\textsc{Y. Li},
Remark on some conformally invariant integral equations: the method of moving spheres,
\textit{J. Eur. Math. Soc. (JEMS)} \textbf{6} (2004), pp. 153--180.

\bibitem[LZ95]{lz1995}
\textsc{Y. Li, M. Zhu}, 
Uniqueness theorems through the method of moving spheres, 
\textit{Duke Math. J.} \textbf{80} (1995), pp. 383--417.

\bibitem[LZ03]{lz2003}
\textsc{Y. Li, M. Zhu}, 
Liouville-type theorems and Harnack-type inequalities for semilinear elliptic equations, 
\textit{J. Anal. Math.} \textbf{90} (2003), pp. 27--87.

\bibitem[Lie83]{l1983}
\textsc{E.H. Lieb},
Sharp constants in the Hardy--Littlewood--Sobolev and related inequalities,
\textit{Ann. Math.} \textbf{118} (1983), pp. 349--374.

\bibitem[LL01]{liebloss2001}
\textsc{E.H. Lieb, M. Loss},
Analysis,
\textit{2nd ed. Graduate studies in Mathematics {\bf 14}, Providence, RI: American Mathematical Sociery, 2001}.

\bibitem[NN15]{NgoNguyen2015}
\textsc{Q.A. Ng\^o, V.H. Nguyen},
{Sharp reversed Hardy--Littlewood--Sobolev inequality on $\Rset^n$}, conditionally accepted in \textit{Israel Journal of Mathematics}, arXiv:1508.02041.

\bibitem[NN15H]{NgoNguyen2015Heisenberg}
\textsc{Q.A. Ng\^o, V.H. Nguyen},
{Sharp reversed Hardy--Littlewood--Sobolev inequality on the Heisenberg group}, \textit{in preparation}, 2016.

\bibitem[Ser71]{s1971}
\textsc{J. Serrin}, 
A symmetry problem in potential theory, 
\textit{Arch. Ration. Mech. Anal.} \textbf{43} (1971), pp. 304--318.

\bibitem[Sob38]{sobolev1938}
\textsc{S.L. Sobolev},
On a theorem of functional analysis,
\textit{Math. Sb. (N.S.)} {\bf 4} (1938) 471-479. English transl. in \textit{Amer. Math. Soc. Transl. Ser. 2} {\bf 34} (1963) 39-68.

\bibitem[SW58]{st1958}
\textsc{E.M. Stein, G. Weiss},
Fractional integrals in $n$-dimensional Euclidean space,
\textit{J. Math. Mech.} {\bf 7} (1958) 503-514.

\bibitem[Xu05]{xu2005}
\textsc{X. Xu},
Exact solutions of nonlinear conformally invariant integral equations in $\mathbf R^3$,
\textit{Adv. Math.} \textbf{194} (2005), pp. 485--503.

\bibitem[Zhu14]{zhu2014}
\textsc{M. Zhu},
Prescribing integral curvature equation,
\textit{preprint}, arXiv:1407.2967v2.

\bibitem[Zhu15]{zhu2015}
\textsc{M. Zhu},
A short proof for reversed Hardy--Littlewood--Sobolev inequality,
\textit{preprint}.

\end{thebibliography}
\end{document}